\documentclass[12pt]{amsart}

\usepackage{amssymb}
\usepackage{mathtools}
\usepackage[english]{babel}
\usepackage{epsfig}
\usepackage{mathptmx}
\usepackage{amsmath,amsfonts,amssymb}
\usepackage{mathtools}
\usepackage{enumitem}
\usepackage{mathrsfs}
\DeclareMathAlphabet{\mathcal}{OMS}{cmsy}{m}{n} %for having \mathcal
\usepackage[all]{xy}
\usepackage{graphicx}
\usepackage{latexsym}
\usepackage{verbatim}
\usepackage[usenames,dvipsnames]{xcolor}
\usepackage[normalem]{ulem}
\usepackage{cancel}
\usepackage{tikz-cd}\usetikzlibrary{cd}

\def\HYPER{\relax}

\ifdefined\HYPER
\usepackage[hidelinks, pdfstartview=FitB, pdfpagemode=UseNone]{hyperref}
\else
\renewcommand{\href}[2]{\relax}
\renewcommand{\url}[1]{#1}
\fi
\usepackage{marvosym}
\makeatletter
\@namedef{subjclassname@2020}{\textup{2020} Mathematics Subject Classification}
\makeatother

\usepackage[text={16.5cm, 22.25cm}, asymmetric]{geometry}
\numberwithin{equation}{section}

\def\Re{\mathop{\mathrm{Re}}}
\def\Im{\mathop{\mathrm{Im}}}

\def\D{\mathbb D}

\def\H{\mathbb{H}}

\def\R{\mathbb R}

\def\C{\mathbb C}
\def\Hol{{\sf Hol}}
\def\Aut{{\sf Aut}}

\def\N{\mathbb N}

\def\id{{\sf id}}

\def\arctanh{\mathop{\mathrm{arctanh}}}

\newcommand{\widesim}[1]{
	\mathrel{\lower.36ex\hbox{$\overset{#1}{\scalebox{1.25}[1]{$\sim$}}$}}
}

\newtheorem{theorem}{Theorem}%[section]
\newtheorem{lemma}{Lemma}[section]
\newtheorem{proposition}[lemma]{Proposition}
\newtheorem{corollary}[theorem]{Corollary}

\newtheorem{theorem*}{Theorem}
\theoremstyle{definition}
\newtheorem{definition}[lemma]{Definition}

\newtheorem{addendum}{Addendum}
\theoremstyle{remark}
\newtheorem{remark}[lemma]{Remark}

\numberwithin{equation}{section}

\newcommand{\UH}{\mathbb{H}}
\newcommand{\UD}{\mathbb{D}}
\newcommand{\UC}{\partial\UD}

\newcommand{\Real}{\mathbb{R}}
\newcommand{\Natural}{\mathbb{N}}
\DeclareMathOperator{\di}{d\!}
\newcommand{\anglim}{\angle\lim}
\newcommand{\BH}{{\mathrm{B}_\UH}} %(closed) hyperbolic disk

\renewcommand{\ge}{\geqslant}
\renewcommand{\le}{\leqslant}
\renewcommand{\geq}{\geqslant}
\renewcommand{\leq}{\leqslant}

\newcommand{\weaklyto}{\xrightarrow{\text{weakly\,}}}

\newcommand{\proofof}[1]{{\fontseries{bx}\fontshape{it}\selectfont Proof of #1}}

\newcommand{\StepP}[1]{\medskip\noindent{\textsc{Proof of} #1.}}
\newcommand{\StepPm}[1]{\medskip\noindent{\textsc{Proof of} #1~}}

\newcommand{\StepG}[1]{\medskip\noindent{\textsc{#1}}}

\newcommand{\df}[1]{{\fontseries{bx}\fontshape{it}\selectfont #1}}

\newenvironment{arablist}{\begin{enumerate}[label={\rm (\arabic*)}, ref={\hbox{(\arabic*)}}, left=.7em]}{\end{enumerate}}
\newcounter{alphlistcounter}
\newenvironment{alphlist}{\begin{enumerate}[label={\bf (\alph*)}, ref={\hbox{(\alph*)}}, left=0em]%%
	\setcounter{enumi}{\value{alphlistcounter}}%
	\everydisplay{\makeatletter\def\@eqnum{\normalfont(\theequation)}\makeatother}}{\end{enumerate}}
\newenvironment{alphlistP}{\begin{enumerate}[label={\bf (\alph*')}, ref={\hbox{(\alph*')}}, left=0em]%
	\setcounter{enumi}{\value{alphlistcounter}}%
	\everydisplay{\makeatletter\def\@eqnum{\normalfont(\theequation)}\makeatother}}{\end{enumerate}}
\newenvironment{alphlistPP}{\begin{enumerate}[label={\bf (\alph*'')}, ref={\hbox{(\alph*'')}}, left=0em]%
	\setcounter{enumi}{\value{alphlistcounter}}%
	\everydisplay{\makeatletter\def\@eqnum{\normalfont(\theequation)}\makeatother}}{\end{enumerate}}
\newenvironment{romlist}{\begin{enumerate}[label={\bf (\hskip.07em\roman*\hskip.07em)}, ref={\hbox{(\hskip.05em\roman*\hskip.05em)}}, left=0em]%
	\everydisplay{\makeatletter\def\@eqnum{\normalfont(\theequation)}\makeatother}}{\end{enumerate}}
\DeclareRobustCommand{\SkipTocEntry}[5]{}

\def\hypdist{{\mathrm k}}%Hyperbolic distance
\def\hyplen{{\ell}}%Hyperbolic length
\newcommand{\dhyp}{\mathop{\mathrm D_h}\!}
\newcommand{\hsegment}[1]{{[#1]_{h}}} %Hyperbolic segment
\newcommand{\Hbspace}{\mathcal{H}(\varphi)} %De Branges-Rovnyak space for \varphi

\allowdisplaybreaks
\begin{document}

\title[Hyperbolic distortion and conformality at the boundary]{Hyperbolic distortion and conformality at the boundary}

\author[P. Gumenyuk]{Pavel Gumenyuk}
\address{P. Gumenyuk: Department of Mathematics, Politecnico di Milano, via E. Bonardi 9, 20133 Milan, Italy.}
\email{pavel.gumenyuk@polimi.it}

\author[M. Kourou]{Maria Kourou$^{\S}$}
\thanks{$^{\S}$Partially supported by the Alexander von Humboldt Foundation.}
\address{M. Kourou: Department of Mathematics, Julius-Maximilians University of W\"urzburg, Emil Fischer Strasse 40, 97074 W\"urzburg, Germany.}
\email{maria.kourou@uni-wuerzburg.de}

\author[A. Moucha]{Annika Moucha$^{\S}$}
%\thanks{$^{\S}$Partially supported by the Alexander von Humboldt Foundation.}
\address{A. Moucha: Department of Mathematics, Julius-Maximilians University of W\"urzburg, Emil Fischer Strasse 40, 97074 W\"urzburg, Germany.}
\email{annika.moucha@uni-wuerzburg.de}

\author[O. Roth]{Oliver Roth}
\address{O. Roth: Department of Mathematics, Julius-Maximilians University of W\"urzburg, Emil Fischer Strasse 40, 97074 W\"urzburg, Germany.} \email{roth@mathematik.uni-wuerzburg.de}

\date\today

%\thanks{}

\begin{abstract}
	We characterize two classical types of conformality of a holomorphic self-map of the unit disk at a boundary point~---   existence of a finite angular derivative in the sense of Carath\'eodory and the weaker property of angle preservation~---  in terms of the non-tangential asymptotic behaviour of the hyperbolic distortion of the map. These characterizations are given purely with reference to the intrinsic metric geometry of the unit disk. In particular, we relate the classical Julia\,--Wolff\,--\,Carath\'eodory theorem with the case of equality in the Schwarz\,--\,Pick lemma at the boundary. We also provide an operator-theoretic characterization of the existence of a finite angular derivative based on Hilbert space methods. As an application we study  the backward dynamics of discrete dynamical systems induced by holomorphic self-maps,  and characterize the regularity of the associated pre-models in terms of a Blaschke-type condition involving the hyperbolic distortion along regular backward orbits.
\end{abstract}

\maketitle

\tableofcontents

\newpage

\section{Introduction} \label{sec:Intro}
In this paper we relate the Julia\,--\,Wolff\,--\,Carath\'eodory Theorem to the case of asymptotic equality at the boundary in the Schwarz\,--\,Pick lemma.
We deal with two classical notions of boundary conformality   for holomorphic self-maps of the unit disk~$\UD\subseteq\C$:

\begin{enumerate}[ref=label]
	\item[($\text{C}_{\text{w}}$)]\label{IT:NTConformality}
	The non-tangential conformality, in the sense of  angle preservation, at a boundary point~${\sigma\in\UC}$;\smallskip
	\item[($\text{C}_{\text{s}}\;$)]\label{IT:AngDer}
	The existence of a finite angular derivative at a boundary point $\sigma\in\UC$.
\end{enumerate}
We refer to \hyperref[IT:NTConformality]{($\text{C}_{\text{w}}$)} as \textbf{conformality at $\sigma$ in the weak sense} and to \hyperref[IT:AngDer]{($\text{C}_{\text{s}}$)} as \textbf{conformality at $\sigma$ in the strong sense} (see Section \ref{SEC:MainResults} for the precise definitions). It should be mentioned that both types of boundary conformality of a holomorphic self-map ${\varphi:\UD\to\UD}$ guarantee  that the angular limit $\varphi(\sigma):=\anglim_{z\to\sigma}\varphi(z)$ of~$\varphi$ at~$\sigma$ does exist and belongs to~$\UC$.

The notion of weak conformality at the boundary is a fundamental and well-developed concept in complex analysis. Some authors call it  \textit{isogonality} (see \cite{Pommerenke:BB}) or \textit{semiconformality} (see \cite{BK2022, JO1977,Ostrowski1937,JO1977}). In contrast to strong conformality, it  has a clear geometric meaning:  every smooth curve in $\D$ which lands non-tangentially at $\sigma$ is mapped  onto a smooth curve landing at some point $\omega\in\UC$, where~$\omega$ is independent of the choice of the curve, and  the angles between two such curves are preserved, see e.g. \cite[Sections 4.3,\,11.3 and\,\,11.4]{Pommerenke:BB} or \cite[Sections V.5 and\,\,V.6]{GarnettMarshall2005}. For this reason, this concept arises naturally in geometric function theory, e.g.~for studying the well-known angular derivative problem for conformal mappings \cite{BK2022, GarnettMarshall2005,JO1977, Pommerenke:BB,RodinWarsch}. It has found various recent applications, for instance in the study of discrete and continuous iteration in hyperbolic domains, in particular for the investigation of so-called pre-models associated to repulsive fixed points, see  \cite{Poggi-Corradini,Poggi-Corradini2003} and \cite[Section 13.2]{BCD-Book}, and in connection with commuting self-maps \cite[Section 8]{commuting-secondpaper}.

The notion of strong conformality at the boundary, that is,
the existence of a finite angular derivative in the sense of Carath\'eodory,  is also rooted in classical geometric function theory. It is relevant for several areas of recent research such as composition operators \cite{CM1995,Shapiro:book}, 
boundary behaviour of conformal mappings \cite{GarnettMarshall2005, Pommerenke:BB}, continuous semigroups of holomorphic self-maps \cite{BCD-Book}, and discrete iteration~\cite{Marco}.
The angular derivative has also been considered  from an operator-theoretic point of view (see \cite{AMY,FM1,FM2,GMR,Sarason1988,Sarason1994}) and for holomorphic functions of several complex variables (\cite{Abate1990,Abate1998,Agler2012,AF2024,Herve1963,Rudin1980}).

Starting with the classical works of Julia, Wolff and Carath\'eodory in the 1920s,  the concept of angular derivative has been closely tied to a boundary version of the Schwarz lemma, known as Julia's lemma or the Julia\,--\,Wolff lemma; see \cite[Chapter~2]{Marco} for an up-to-date account and \cite{Abate2021,BM2004,FR} for recent multi-point versions of these classical results.
One of the goals of the present paper is to establish a different connection between conformality at the boundary and the Schwarz lemma. This connection is based on a study of the finer aspects of the \emph{cases of equality} in the Schwarz lemma at the boundary.

In fact, the topic of ``Schwarz lemmas at the boundary'', in particular with a focus on the case of equality and in conjunction with rigidity properties of  holomorphic maps, has been widely studied from various perspectives in the past decades.
For maps that are smooth up to the boundary, the underlying idea can be traced back to the famous work of Loewner \cite[Hilfssatz I]{Loe23} introducing the celebrated Loewner differential equation. The basic result in this direction without any assumption about regularity at the boundary was obtained by Burns and Krantz in their pioneering paper~\cite{BurnsKrantz1994}. Their work has stimulated much activity in recent years in boundary rigidity results for holomorphic mappings in one and several complex variables, see e.g.~\cite{BaraccoZaitsevZampieri2006,BKR2024,FDO,Chelst2001,Dubinin2004,
	EJLS2014,LiuTang2016,Osserman2000,Shoikhet2008,TauVla2001,Zimmer2022}.

The main objective of the present paper is to demonstrate that each of the  conditions \hyperref[IT:NTConformality]{($\text{C}_{\text{w}}$)} and~\hyperref[IT:AngDer]{($\text{C}_{\text{s}}$)} corresponds exactly to  a certain natural ``non-tangential'' case of equality in the invariant Schwarz lemma at the boundary point~$\sigma$. This leads to conformally invariant characterizations of both types of boundary conformality entirely in terms of the natural metric geometry of the disk. Thereby,  our results link the classical Julia\,--\,Wolff\,--\,Carath\'eodory theorem on the angular derivative of  holomorphic self-maps $\varphi$ of the unit disk $\UD$ with
the circle of ideas around the boundary Schwarz lemma of  Burns and Krantz.  We also show that conformality in the strong sense corresponds precisely to~a natural ``non-tangential'' case of equality at the boundary in the Cauchy\,--\,Schwarz inequality for the de Branges\,--\,Rovnyak space $\Hbspace$ associated to $\varphi$. In fact, we give an independent proof of this  equivalence by  operator-theoretic methods. As an application of our approach, we investigate a problem for the discrete dynamical system induced by iterating  a holomorphic self-map $\varphi : \D \to \D$. Specifically, we characterize the regularity of the pre-model for~$\varphi$ at a repulsive fixed point ${\sigma\in\UC}$ in terms of the hyperbolic distortion of~$\varphi$ evaluated along a regular backward orbit. This extends the main result of the recent paper~\cite{our-in-RMI} on holomorphic semiflows in~$\UD$ to the discrete setting.

\section{Main results} \label{SEC:MainResults}
The hyperbolic metric $\lambda_\UD(z) \, |\!\di z|$  of the unit disk~$\UD:=\{z\in\C\colon |z|<1\}$ is defined by
$$ \lambda_{\D}(z):=\frac{2}{1-|z|^2} \, ,$$
and we denote the associated hyperbolic distance between two points $z,w \in \D$ by
$\hypdist_\UD(z,w)$.
The invariant Schwarz lemma or Schwarz\,--\,Pick lemma is concerned with the set $\Hol(\UD,\UD)$ of holomorphic self-maps of $\D$. It says that every $\varphi \in \Hol(\D,\D)$
is a weak contraction of the metric space $(\UD,\hypdist_\UD)$, that is
\begin{equation} \label{EQ:SchwarzPickTwoPoints}
	\hypdist_{\UD}\big(\varphi(z),\varphi(w) \big) \le \hypdist_{\UD} \big( z,w\big),\quad z,w \in \UD.
\end{equation}
Moreover,  equality holds for two distinct points $z,w \in \UD$  --- and hence for all $z,w\in \UD$ --- if and only if $\varphi$ belongs to the group $\Aut(\D)$ of all conformal automorphisms (or biholomorphisms) of~$\D$. Equivalently,
the \emph{hyperbolic distortion} $\dhyp \varphi : \D \to [0,+\infty)$ defined by
\begin{equation} \label{EQ:HypDistDef}
	\dhyp \varphi(z):=\lim \limits_{w \to z} \frac{\hypdist_\UD\big(\varphi(z),\varphi(w)\big)}{\hypdist_\UD(z,w)}=\left( 1-|z|^2 \right) \frac{|\varphi'(z)|}{1-|\varphi(z)|^2}
\end{equation}
satisfies
\begin{equation} \label{EQ:SchwarzPickOnePoint}
	\dhyp \varphi(z) \le 1
\end{equation}
for all $z \in \D$ with equality for one --- and hence for every --- ${z \in \D}$ if and only if ${\varphi \in \Aut(\D)}$; see \cite[Def.~5.1]{BM2004}.

As usual, we use the symbol $\,\,\anglim\,\,$ to denote the \emph{non-tangential} or \emph{angular limit}, see e.g. \cite[Section 2.2]{Marco}.
The concrete point of departure of this paper has been the following question posed in~\cite[Problem~5.2]{FDO}: \textit{let $\varphi\in\Hol(\UD,\UD)$, $\sigma\in\partial\UD$, and suppose  that
	\begin{equation}\label{EQ_anglim-abs-hder=1}
		\anglim_{z\to\sigma} \dhyp \varphi(z)~=~1.
	\end{equation}
	Does this imply that $\varphi$ has an angular limit at~$\sigma$?}
\smallskip

Our first main result implies an affirmative answer to the above question and, in fact, tells much more. In order to properly formulate this result, we need to introduce some terminology.

\begin{definition}\label{DF_conformality-at-the-boundary}
	Let $\sigma\in\UC$. We call $\varphi\in\Hol(\UD,\UD)$ \df{conformal at~$\sigma$ in the strong sense} if $\varphi$ possesses a finite angular derivative in the sense of Carath\'eodory at~$\sigma$, i.e.~if  there exists~$\omega\in\UC$ such that the angular limit
	\begin{equation}\label{EQ_angular-Caratheodory}
		\varphi'(\sigma):=\anglim_{z\to\sigma}\frac{\varphi(z)-\omega}{z-\sigma}
	\end{equation}
	is finite.
	We call $\varphi\in\Hol(\UD,\UD)$ \df{conformal at~$\sigma$ in the weak sense} if\\[-2mm]
	\begin{equation}\label{EQ_contact-point-condition}
		\varphi(\sigma):=\anglim_{z\to\sigma}\varphi(z)\,\in\,\UC
	\end{equation}
	and
	\begin{equation}\label{EQ_isogonality-for-selfmap}
		\anglim_{z\to\sigma}\arg\frac{1-\overline{\varphi(\sigma)}\varphi(z)}{1-\overline\sigma z}=0.
	\end{equation}
      \end{definition}
We emphasize that the limit value $\varphi(\sigma)$ in~\eqref{EQ_contact-point-condition} is required to lie on the unit circle~$\UC$.
\begin{remark} If $\varphi\in \Hol(\D,\D)$ is conformal at~$\sigma$ in the strong sense, then~$\varphi$ is clearly also conformal at~$\sigma$ in the weak sense, with ${\varphi(\sigma)=\omega}$. The converse is not true, see \cite[Example 2.2]{Zorboska2015}  and also Section~\ref{SUB:AngDerVSHypDist} below.  The angular limit $\varphi'(\sigma)$  in~\eqref{EQ_angular-Caratheodory} exists (finitely or infinitely) for any~${\sigma,\omega\in\UC}$ and any ${\varphi\in\Hol(\UD,\UD)}$; see e.g. \cite[Theorem~2.3.2]{Marco}. Moreover,  $\varphi$ is conformal at~$\sigma$ in the strong sense if and only if~\eqref{EQ_contact-point-condition} holds and the angular limit ${\anglim_{z\to\sigma}\varphi'(z)}$ exists and is finite, see e.g. \cite[Proposition~2.3.1]{Marco}, \cite[Proposition~4.7]{Pommerenke:BB}, or \cite[Section 4.2]{Shapiro:book}. Geometrically, condition~\eqref{EQ_isogonality-for-selfmap} means that for smooth curves in $\UD$ tending non-tangentially to~$\sigma$, the angles at which they meet $\UC$ are preserved by~$\varphi$.
\end{remark}

We are now prepared to state our first main result.
\begin{theorem}\label{TH-main1}
	Suppose $\varphi\in\Hol(\UD,\UD)$ and $\sigma \in \partial \UD$. Then the following conditions are pairwise equivalent:
	\begin{alphlist}
		\item\label{IT_TH-main1:hder-to-1}
		${\dhyp \varphi(r\sigma)\to1}$ as ${r\to1^-}$.\\[-1mm]
		\item\label{IT_Th-main1:hdiffquo-to-1}
		$\displaystyle \angle \lim \limits_{z,w \to \sigma} \frac{\hypdist_{\UD}\big(\varphi(z),\varphi(w) \big)}{\hypdist_\UD \big(z,w \big)} = 1$.\\[0mm]
		\item\label{IT_TH-main1:semireg-cntct-pnt}
		The self-map $\varphi$ is conformal at $\sigma$ in the weak sense.
	\end{alphlist}
\end{theorem}

In condition~\ref{IT_Th-main1:hdiffquo-to-1}, we adopt the convention that we identify the hyperbolic difference quotient $\hypdist_\UD(\varphi(z),\varphi(w))/\hypdist_\UD(z,w)$ for $w=z$  with $\dhyp \varphi(z)$. This is in accordance with \eqref{EQ:HypDistDef}. Theorem~\ref{TH-main1} shows that equality in the Schwarz\,--\,Pick inequalities~\eqref{EQ:SchwarzPickTwoPoints} and~\eqref{EQ:SchwarzPickOnePoint} is attained at a boundary point ${\sigma\in\UC}$ in the sense that the hyperbolic distortion
$\dhyp \varphi(z)$ tends to~$1$ resp.~the
\textit{quotient} $\,\hypdist_\UD(\varphi(z),\varphi(w))/\hypdist_{\UD}(z,w)$ tends to~$1$ as $z$ and $w$ approach $\sigma$ non-tangentially if and only if the self-map~$\varphi$ is conformal at $\sigma$ in the weak sense. We wish to emphasize that one of the main  contributions of the present paper lies in proving that condition \ref{IT_TH-main1:hder-to-1} implies condition \ref{IT_TH-main1:semireg-cntct-pnt}.

In fact, there are equivalent analogues to each of the three conditions in the above theorem. In order to state them, we need the following definition: a sequence $(z_n)$ in~$\D$ is said to be of \df{bounded (hyperbolic) step}\label{bounded-hyperbolic-step} if there is a constant ${M>0}$ such that $\hypdist_\UD\big(z_n,z_{n+1}\big) \le M$ for all  $n\in\N:=\{1,2,\dots\}$.

\begin{addendum}[\textsl{Further conditions equivalent to conformality at the boundary in the weak sense}] \label{ADD:TH-main1}
	A self-map $\varphi \in \Hol(\UD,\UD)$ is conformal at the point $\sigma \in \partial \D$ in the weak sense if and only if one~--- and hence each~--- of the following conditions holds:
	\begin{alphlistP}
		\item\label{IT_FEQ_a} There is a sequence $(z_n)\subseteq\D$ of bounded step converging non-tangentially to $\sigma$ such that
		$$
		\lim \limits_{n \to \infty}\dhyp \varphi(z_n) =1\, .$$
		\smallskip%
		\item\label{IT_FEQ_c} There is a sequence  $(z_n)\subseteq\D$ of bounded step converging non-tangentially to $\sigma$ such that
		$$
		\frac{\hypdist_{\UD}\big(\varphi(z_n),\varphi(z_{n+1}) \big)}{\hypdist_\UD \big( z_n,z_{n+1} \big)} \to 1 \quad\text{as~$~n\to+\infty$}\,.
		$$
		\smallskip%
		\item\label{IT_FEQ_b} The following two angular limits exist:
		$$
		\angle \lim \limits_{z \to \sigma} \varphi(z) \in \partial \D \quad \text{and} \quad \angle \lim \limits_{z \to \sigma} \arg \varphi'(z) \in \Real\, .
		$$
	\end{alphlistP}
\end{addendum}

\begin{remark} \label{REM:TH-main1}
	The equivalence of \ref{IT_FEQ_b} and \ref{IT_TH-main1:semireg-cntct-pnt}  follows from the results obtained by Yamashita in \cite{Yamashita}, who considered the more general situation that $\varphi : \D \to \C$ is holomorphic and non-constant, see Remark~\ref{RM_Yamashita} for the details. Since, in general,~$\varphi'$ is \textit{not} required to be non-vanishing in~$\UD$, it is however necessary to explain in which sense the existence of $\anglim_{z \to \sigma} \arg \varphi'(z)$ is to be understood. This part of condition~\ref{IT_FEQ_b} means that for each Stolz region ${S\subseteq\UD}$ with vertex at~$\sigma$, the following holds: there exists ${\varepsilon>0}$ such that ${\varphi'(z)\neq0}$ for all ${z\in S_\varepsilon}:={\{z\in S\colon |z-\sigma|<\varepsilon\}}$, and the limit of $\arg \varphi'(w)$ as ${S_\varepsilon\ni z\to\sigma}$ exists in the topology of ${\Real/(2\pi\mathbb Z)}$.
\end{remark}

\begin{remark}[Weak Conformality and the Visser\,--\,Ostrowski condition] \label{REM:WeakConformalityVSVisserOstrowski}
		Under the additional assumptions that $\anglim_{z\to\sigma}\varphi(z)=:\varphi(\sigma)$ exists and  ${\varphi(\sigma)\in\UC}$, the conformality of~$\varphi \in \Hol(\D,\D)$ at~$\sigma$ in the weak sense is equivalent to the classical \textbf{Visser\,--\,Ostrowski} condition:
		\begin{equation} \label{EQ:VisserOstrowski}
			\anglim \limits_{z \to \sigma}
			\frac{(z-\sigma)\varphi'(z)}{\varphi(z)-\varphi(\sigma)}=1 \, .
		\end{equation}
		This is a rather simple fact, which immediately follows, e.g., from Lemma~\ref{LM_isogonality} included in the Appendix. Up to our best knowledge, it has not been, however, known before. It is also worth mentioning that without the hypotheses  $\varphi(\UD)\subset\UD$ and $\varphi(\sigma)\in\UC$, condition~\eqref{EQ:VisserOstrowski} is weaker than isogonality, i.e. it does not imply the existence of $\anglim_{z\to\sigma}\arg\varphi'(z)$; see Section~\ref{SUB:terminology} and Remark~\ref{RM_Yamashita} for more details.

Note also that the above Visser\,--\,Ostrowski condition easily implies that $\dhyp \varphi(z)$ tends to~$1$ if ${z \to \sigma}$ radially, i.e. condition~\ref{IT_TH-main1:hder-to-1} in Theorem~\ref{TH-main1}.
\end{remark}

Using Theorem~\ref{TH-main1}, we are further able to characterize conformality in the strong sense of a self-map $\varphi \in \Hol(\UD,\UD)$ in terms of the non-tangential boundary behaviour of the hyperbolic distortion~$\dhyp \varphi(z)$. In particular, this gives a complete answer to~\cite[Problem~5.3]{FDO}. Similarly to Theorem~\ref{TH-main1},  we also provide an equivalent two-point condition in terms of the hyperbolic distances. Moreover, we can add an operator-theoretic condition in terms of the non-tangential boundary behavior of the reproducing kernel functions~$k_z^\varphi$ in the de Branges\,--\,Rovnyak space induced by~$\varphi$ (for the precise definitions and terminology used in condition~\ref{IT_TH-main2:operator-theory-liminf1} of the statement, see Section~\ref{SEC:Operatortheory}).
\newpage
\begin{theorem}\label{TH-main2}
	Let $\varphi\in\Hol(\UD,\UD)$ and $\sigma\in\partial\D$. Then the following  conditions are pairwise equivalent:
	\begin{alphlist}
		\item\label{IT_TH-main2:int-conv}
		$\displaystyle\int_0^1\big(1-\dhyp \varphi(r\sigma)\big)\,\lambda_\UD(r\sigma) \di r~<+\infty~$.\\[2mm]
		\item\label{IT_TH-main2:BetsakosKaramanlis-angular}
		$\displaystyle \angle \lim \limits_{z,w \to \sigma}\big(\,\hypdist_\UD\big(z,w\big)-\hypdist_\UD\big(\varphi(z),\varphi(w)\big)\,\big)\,=\,0$.
		\smallskip%

		\item\label{IT_TH-main2:operator-theory-liminf1}
		$\displaystyle \angle \lim_{z,w\to \sigma}\frac{\left|\langle k_z^\varphi, k_w^\varphi \rangle_\varphi \right\vert}{\left\Vert k_z^\varphi \right\Vert_\varphi  \left\Vert k_w^\varphi \right\Vert_\varphi} = 1$.\\[2mm]
		\item\label{IT_TH-main2:regular-cntct-pnt}
		The self-map $\varphi$ is conformal at the point~$\sigma~$ in the strong sense.
	\end{alphlist}
\end{theorem}

As with Theorem~\ref{TH-main1}, one can add further  conditions equivalent to those listed in Theorem~\ref{TH-main2}.

\begin{addendum}[\textsl{Further conditions equivalent to conformality at the boundary in the strong sense}]  \label{ADD:TH-main2Condition(d)}
	A self-map $\varphi \in \Hol(\UD,\UD)$ is conformal at the point $\sigma \in \partial \D$ in the strong sense if and only if  one~--- and hence both~--- of the following conditions holds:\smallskip
	\setcounter{alphlistcounter}{1}
	\begin{alphlistP}
		\item \label{IT_TH-main2:BetsakosKaramanlis-radial}
		$\hypdist_\UD\big(r\sigma,s\sigma\big)-\hypdist_\UD\big(\varphi(r\sigma),\varphi(s\sigma)\big) \to 0~$ as $~(0,1)\ni r,s\to1$.
		\medskip

		\item \label{IT_TH-main2:operator-theory}
		$\displaystyle \angle \liminf_{z,w\to \sigma}\frac{\left|\langle k_z^\varphi, k_w^\varphi \rangle_\varphi \right\vert}{\left\Vert k_z^\varphi \right\Vert_\varphi  \left\Vert k_w^\varphi \right\Vert_\varphi} > 0$.
	\end{alphlistP}
\end{addendum}
\begin{remark}
The major aspect of Theorem~\ref{TH-main2} is probably the equivalence between  \ref{IT_TH-main2:int-conv} and \ref{IT_TH-main2:regular-cntct-pnt}.
A~proof of  this  equivalence can also be found in a recent preprint of O.~Ivrii and M.~Urba\'{n}ski~\cite[Theorem~B.1]{Ivr}.
The approach developed in~\cite{Ivr} is of strong geometric flavour; in particular, it uses the relation of the hyperbolic distortion~$\dhyp\varphi$ to the geodesic curvature of the image ${\varphi([0,\sigma))}$ of the geodesic ray ${[0,\sigma)}$. Our approach is different, and dual in the sense that we work with the \emph{preimage}~$\gamma$ of ${[0,\sigma)}$. As we shall see in Section~\ref{sec_proof-main2}, the key role in our proof that condition~\ref{IT_TH-main2:int-conv} implies strong conformality is played by  several properties of~$\gamma$ which follow directly from the weak conformality, see Proposition~\ref{PR_Real_axis_in_image}. In fact, apart from the operator-theoretic aspect, Theorem~\ref{TH-main2} appears to be a fairly direct consequence of the more general Theorem~\ref{TH-main1} and the techniques we employ in its proof.
\end{remark}
\begin{remark}
	The equivalence \ref{IT_TH-main2:int-conv}~$\Longleftrightarrow$~\ref{IT_Th-main1:hdiffquo-to-1} in Theorem~\ref{TH-main2} is a statement purely in terms of the path-metric space $(\UD,\hypdist_\UD)$. The same remark applies to conditions~\ref{IT_TH-main2:int-conv} and~\ref{IT_Th-main1:hdiffquo-to-1} in Theorem~\ref{TH-main1}.  We will discuss these metric aspects of our results in more detail in Section~\ref{SUB:MetricAspects}.
\end{remark}

\begin{remark}
	\label{REM:TH-main2Condition(d)}
	Condition \ref{IT_TH-main2:BetsakosKaramanlis-radial} is closely related to results recently obtained by Betsakos and Karamanlis \cite{BK2022} in their work on the angular derivative problem for univalent mappings.
	We discuss this in greater detail in Section~\ref{SUB:AngDerivativeConformalMaps}. A very non-trivial fact contained in Theorem~\ref{TH-main2} is that condition \ref{IT_TH-main2:BetsakosKaramanlis-radial} is equivalent to the \textit{a priori} much weaker condition~\ref{IT_TH-main2:int-conv}. Indeed,~\ref{IT_TH-main2:int-conv} can be rewritten in integrated form as
	$$
	\hypdist_\UD(r\sigma,s\sigma)-\hyplen_\UD\big(\varphi(\hsegment{s\sigma,r\sigma})\big)\to0 \quad\text{as~$~(0,1)\ni r\ge s\to1$} \, ,
	$$
	where $\hyplen_\UD(\cdot)$ stands for the hyperbolic length of a smooth curve in~$\UD$ and $\hsegment{z,w}$ denotes the hyperbolic segment joining~$z$ and~$w$; clearly,
	$$
	\hyplen_\UD\big(\varphi(\hsegment{s\sigma,r\sigma})\big)\ge \hyplen_\UD\big(\hsegment{\varphi(s\sigma),\varphi(r\sigma)}\big) = \hypdist_\UD\big(\varphi(r\sigma),\varphi(s\sigma)\big).
	$$
\end{remark}
\begin{remark}
 Condition~\ref{IT_TH-main2:BetsakosKaramanlis-angular} in Theorem~\ref{TH-main2}, as compared to its counterpart in Theorem~\ref{TH-main1}, represents another variant of asymptotic equality in the Schwarz\,--\,Pick inequality~\eqref{EQ:SchwarzPickTwoPoints}. Note that \textit{a priori} it does not seem clear that one of these variants is stronger than the other. This is a non-trivial consequence of Theorems~\ref{TH-main1} and~\ref{TH-main2}.
It should be noted however that the equivalence \ref{IT_TH-main2:BetsakosKaramanlis-angular} $\Longleftrightarrow$ \ref{IT_TH-main2:regular-cntct-pnt} in Theorem~\ref{TH-main2} is fairly easy to establish, cf.~Remark \ref{REM:TH-main2(c)VS(d)} for details.
\end{remark}

\begin{remark}
   Condition~\ref{IT_TH-main2:operator-theory-liminf1} in Theorem \ref{TH-main2}  can be regarded as the case of
    non-tangential boundary equality in the Cauchy\,--\,Schwarz inequality for the de Branges\,--\,Rovnyak space of $\varphi$.  In fact, we shall show that the \textit{a priori} weaker condition~\ref{IT_TH-main2:operator-theory} is equivalent to~\ref{IT_TH-main2:regular-cntct-pnt} and a proof of this assertion is given by operator-theoretic methods.
\end{remark}

\begin{remark}
  One may wonder whether the value of the integral
  $$ I(\varphi,\sigma):=\int_0^1\big(1-\dhyp \varphi(r\sigma)\big)\,\lambda_\UD(r\sigma) \di r$$ in condition \ref{IT_TH-main2:int-conv} of Theorem \ref{TH-main2} is related to the value $\varphi'(\sigma)$ of the angular derivative of $\varphi$ at $\sigma$. In fact, it is easy to see that
  $$ I(\varphi,\sigma) \le 2 \log |\varphi'(\sigma)| \, ,$$
  cf.~the proof of Proposition \ref{PR_weak} below. However, an estimate of the form $I(\varphi,\sigma) \ge c \log |\varphi'(\sigma)|+d$ with constants $c$ and $d$ independent of the function $\varphi$ cannot hold. This can be seen by taking the degree two Blaschke products
$$ \varphi_a(z):=z \frac{z-a}{1-a z} \, , \quad 0 \le a<1 \, .$$
Then $|\varphi_a'(1)|=2/(1-a)$, but $I(\varphi_a,1)=\log(1+a)+\log 2$. The authors thank Artur Nicolau for pointing this out to them.
\end{remark}

In the special case of a univalent function mapping $\UD$ onto a symmetric subdomain of~$\UD$, the equivalence between conditions~\ref{IT_TH-main2:int-conv} and~\ref{IT_TH-main2:regular-cntct-pnt} in Theorem~\ref{TH-main2} was recently established in~\cite[Section 6.3]{our-in-RMI}.
In the same work, the equivalence between conditions~\ref{IT_TH-main2:int-conv} and~\ref{IT_TH-main2:regular-cntct-pnt} in Theorem~\ref{TH-main2} was also proved for a special class of univalent holomorphic self-maps of~$\UD$ associated to the backward dynamics of holomorphic semiflows in $\UD$; see \cite[Theorem~1.1]{our-in-RMI}.
Motivated by the latter result, we apply our methods to study and solve a problem arising in the classical field of backward dynamics in discrete iteration of holomorphic self-maps in the disk.

Indeed, the notion of conformality at a boundary point, as introduced in Definition~\ref{DF_conformality-at-the-boundary}, turns out to be highly relevant for the study of holomorphic dynamical systems. In the theory of holomorphic iteration in the unit disk, points ${\sigma\in\UC}$ at which condition~\eqref{EQ_contact-point-condition} holds are referred to as \emph{contact points}, and those of them with ${\varphi(\sigma)=\sigma}$ are called \emph{boundary fixed points}. Furthermore, boundary points $\sigma$ at which $\varphi$ is conformal in the strong sense are known as \emph{regular} contact points and (boundary) \emph{regular} fixed points, respectively. It is worth emphasizing that the role played by boundary regular fixed points is similar, in a sense, to that of interior fixed points (which for a holomorphic self-map ${\varphi\neq\id_\UD}$, clearly, can exist at most one). This principle fully applies to backward dynamics of holomorphic self-maps, the study of which currently attracts much attention, see e.g. \cite{Abate_Raissy_2011, Arosio2017, AB2016, AG2019, AFGK2024, BackwSemigr, Ivr, Poggi-Corradini, Poggi-Corradini2003} and which is intimately linked to repulsive fixed points (i.e. boundary regular fixed points~$\sigma$ with ${\varphi'(\sigma)>1}$).

A renowned result of Poggi-Corradini \cite{Poggi-Corradini,Poggi-Corradini2003} asserts that for any given repulsive fixed point ${\sigma\in\UC}$, there is an associated open subset of the unit disk in which the backward dynamics of~$\varphi$ can be linearized via an (essentially unique) holomorphic map\footnote{Curiously enough, the map constructed by Poggi-Corradini appeared earlier, although implicitly and in a rather special cases (and moreover, restricted to real values of the variable) in the theory of branching processes in connection to some limit theorems~\cite{Galton-Watson,Grey}.} possessing boundary conformality property at the preimage of~$\sigma$, \emph{a priori} in the weak sense. It seems natural to ask in which cases this intertwining map is conformal at the same boundary point in the \emph{strong} sense. We solve this problem with the help of Theorem~\ref{TH-main2}. Our answer is formulated in the following Theorem~\ref{TH-main3} and Corollary~\ref{CR_fromTH-main3}. For the terminology used in the statements of these results and further details from the theory of holomorphic iteration in~$\UD$ we refer the reader to Section~\ref{SEC:ProofTH3}.

\begin{theorem}\label{TH-main3}
	Let $\sigma \in \UC$ be a repulsive fixed point of~$\varphi\in\Hol(\UD,\UD)$   and let $(z_n)\subseteq\UD$ be a regular backward orbit converging to~$\sigma$. The following statements hold:
	\setcounter{alphlistcounter}{0}
	\begin{alphlist}
		\item \label{IT_THmain3:1}
		There exists ${n_0\in\Natural_0:=\Natural\cup\{0\}}$ such that ${\varphi'(z_n)\neq0}$ for all ${n>n_0}$.
		\smallskip%
		\item \label{IT_THmain3:2}
		Moreover, $\dhyp\varphi^{\circ m}(z_{n_0+m})\to\mu$ as ${m\to+\infty}$
		for some $\mu\in(0,1]$.
		\smallskip%
		\item  \label{IT_THmain3:3}
		The pre-model for~$\varphi$ at~$\sigma$ is regular if and only if
		\begin{equation*}\label{EQ_ekte-betingelse-TH-main3}
			\sum_{m=1}^{\infty}\big(\dhyp\varphi^{\circ m}(z_{n_0+m})-\mu\big)~<~+\,\infty.
		\end{equation*}
	\end{alphlist}
\end{theorem}
As a consequence, we obtain another condition equivalent to the regularity of the pre-model, which involves hyperbolic distances instead of the hyperbolic distortion.
\begin{corollary}\label{CR_fromTH-main3}
	Under the hypotheses of Theorem~\ref{TH-main3}, the pre-model for~$\varphi$ at~$\sigma$ is regular if and only if
	\begin{equation*}\label{EQ_fromTH-main3}
		\sum_{n=1}^{\infty}\big(\rho-\hypdist_\UD(z_n,z_{n+1})\big)~<~+\,\infty,\quad \rho:=\lim_{n\to+\infty}\hypdist_\UD(z_n,z_{n+1})
             = 2\arctanh\frac{\varphi'(\sigma)-1}{\big|e^{2i\theta}\varphi'(\sigma)+1\big|},
	\end{equation*}
	where $\theta:=\lim_{n\to+\infty}\arg(1-\overline\sigma z_n)$, for some --- and hence any --- regular backward orbit $(z_n)$ converging to~$\sigma$.
\end{corollary}

\smallskip

This paper is organized as follows.
In Section~\ref{SEC_Properties-Hyperbolic-Distortion}, we establish several basic properties of the hyperbolic distortion of holomorphic self-maps including refined Schwarz\,--\,Pick-type estimates and a version of Lindel\"of's angular limit theorem for hyperbolic distortion. Besides being potentially interesting in their own right, these results play a crucial role in the proofs of our main results, Theorems~\ref{TH-main1} and~\ref{TH-main2}. The proof of Theorem~\ref{TH-main1}  is given in Section~\ref{SEC:ProofTH1}; the proof of Theorem~\ref{TH-main2} is divided into Sections~\ref{sec_proof-main2} (geometric aspects) and~~\ref{SEC:Operatortheory} (operator-theoretic aspects). In Section~\ref{SEC:ProofTH3}, applications to holomorphic dynamics in the unit disk are developed: first we introduce relevant terminology and give necessary preliminaries from the iteration theory; next, we prove our characterization for the regularity of pre-models, i.e. Theorem~\ref{TH-main3} and Corollary~\ref{CR_fromTH-main3}. The paper closes with Section~\ref{SEC:ConcludingRemarks}, in which we explain in more detail the relation of our results with the boundary Schwarz lemma of Burns and Krantz and related results \cite{BKR2024,FDO,BurnsKrantz1994,KRR2007}, with the recent work of Betsakos and Karamanlis \cite{BK2022} and that of Beardon and Minda \cite{BM2023} as well as the those aspects of Theorem~\ref{TH-main1} and Theorem~\ref{TH-main2} which can be expressed purely in terms of metric geometry.
In order to make the paper self-contained, we collect and prove in an appendix a variety of results about weak conformality at the boundary, which seem essentially well-known to the experts, but for which we have not found precise references to the existing literature.

\section{Basic properties of hyperbolic distortion}\label{SEC_Properties-Hyperbolic-Distortion}
In this section we state and prove several basic properties of the hyperbolic distortion
$$ \dhyp \varphi(z)=\left(1-|z|^2 \right) \frac{|\varphi'(z)|}{1-|\varphi(z)|^2}$$
of a holomorphic self-map $\varphi : \D \to \D$ which will be required for the proofs of Theorem \ref{TH-main1} and Theorem \ref{TH-main2}. We derive these properties from the elegant three--point Schwarz\,--\,Pick lemma of Beardon and Minda \cite{BM2004}.

\addtocontents{toc}{\SkipTocEntry}\subsection{Hyperbolic distortion in the unit disk}\label{SUB:hypdistortion}
Throughout this paper we assume familiarity with standard facts of hyperbolic geometry which can be found e.g.~in \cite{Marco,BM2007, BCD-Book}.
Recall that the \emph{pseudo-hyperbolic distance}
\begin{equation}\label{EQ_pseudohyperbolic-distance}
	\varrho_\UD(z,w)=\left| \frac{z-w}{1-\overline{w}z} \right| \,, \qquad z, w \in \UD\,,
\end{equation}
of the unit disk $\UD$ is related to the \emph{hyperbolic distance} $\hypdist_\UD$ by
$$\hypdist_{\UD}\big(z,w\big)=\log \frac{1+\varrho_\UD\big(z,w\big)}{1-\varrho_\UD\big(z,w\big)}\, .$$
Let $\varphi \in \Hol(\UD, \UD)$ and $z,w \in \UD$ with $z\not=w$. Then
\begin{equation*}\label{eq:hypdifquot}
	\varphi^*(z,w)=\frac{1-\overline{w} z}{z-w} \frac{\varphi(z)-\varphi(w)}{1-\overline{\varphi(w)} \, \varphi(z)}
\end{equation*}
is called the \emph{hyperbolic difference quotient} of $\varphi$, see \cite{BM2007}.
Note that if  $\varphi \not \in \Aut(\UD)$, then $\varphi^*(\cdot,v)$ extends, for each fixed ${v \in \UD}$,  to a holomorphic self-map of $\UD$; hence the Schwarz\,--\,Pick lemma implies
\begin{equation}\label{eq:sp-hypdifquot}
	\varrho_\UD\big(\varphi^*(z,v),\varphi^*(w,v)\big) \le \varrho_\UD\big(z,w\big) \,
\end{equation}
for all $z, w \in \UD\setminus\{v\}$. This is the three-point Schwarz\,--\,Pick lemma of Beardon and Minda\,\footnote{They use the hyperbolic distance instead of the pseudo-hyperbolic one.} \cite[Theorem 3.1]{BM2004}. The following sharpening of the Schwarz\,--\,Pick inequality $${\hypdist_\UD(z,w)}-{\hypdist_\UD(\varphi(z),\varphi(w))\ge 0}$$  is a consequence of \eqref{eq:sp-hypdifquot}.

\begin{proposition}\label{prop:Ineq_ApprUniv}
	Let $\varphi \in \Hol(\UD, \UD)$, $w \in \UD$ and $\delta:=1-\dhyp \varphi(w)$. Then, for every $z \in \UD$,
	\begin{equation*}\label{EQ_U_est}
		e^{-\hypdist_\UD(z,w)} \sinh\big(\hypdist_\UD(z,w)\big) \delta  \le\hypdist_{\UD}(z,w) - \hypdist_{\UD}\big(\varphi(z), \varphi(w)\big) \leq  e^{ \hypdist_\UD(z,w)}  \sinh \big( \hypdist_\UD(z,w)\big) \delta \,.
	\end{equation*}
\end{proposition}

\begin{proof} Clearly, we can assume that $\varphi \not\in \Aut(\D)$, so $\delta>0$.
   Fix $z \in \D$. 
    In terms of the  hyperbolic derivative of $\varphi$ at $w$,  which is defined as
	\begin{equation*} \label{EQ:HypDerivativeDef}
		\varphi^h(w):=\lim \limits_{v \to w} \varphi^*(v,w)= \left(1-|w|^2 \right) \frac{\varphi'(w)}{1-|\varphi(w)|^2}  \, ,
	\end{equation*}
	the three-point Schwarz\,--\,Pick inequality \eqref{eq:sp-hypdifquot} for $v \to w$ takes the form
	$$ \varrho_\UD\big(\varphi^*(z,w),\varphi^h(w) \big) \le \varrho_\UD \big(z,w\big) \, . $$
	Since $\varrho_\UD(|a|,|b|) \le \varrho_\UD(a,b)$ and $\dhyp \varphi(a)=|\varphi^h(a)|$ for all $a,b \in \UD$, it follows that
	$$ \left| \frac{|\varphi^*(z,w)|-\dhyp \varphi(w)}{1-\dhyp \varphi(w) |\varphi^*(z,w)|} \right| \le \varrho_\UD(z,w) \, .$$
	This results in two inequalities, namely
	$$ \frac{|\varphi^*(z,w)|-\dhyp \varphi(w)}{1-\dhyp \varphi(w) |\varphi^*(z,w)|}  \le \varrho_\UD(z,w) \quad \text{ and } \quad
	\frac{\dhyp \varphi(w)-|\varphi^*(z,w)|}{1-\dhyp \varphi(w) |\varphi^*(z,w)|}  \le \varrho_\UD(z,w)\, .$$
	Solving for $|\varphi^*(z,w)|$ and using $|\varphi^*(z,w)|= \varrho_\UD\big(\varphi(z),\varphi(w) \big)/ \varrho_\UD(z,w)$,  we obtain
	$$ \frac{\dhyp \varphi(w)-\varrho_\UD(z,w)}{1-\dhyp \varphi(w) \varrho_\UD(z,w)} \varrho_\UD(z,w) \le
	\varrho_\UD\big(\varphi(z),\varphi(w) \big)  \le \frac{\varrho_\UD(z,w)+\dhyp \varphi(w)}{1+\varrho_\UD(z,w) \dhyp \varphi(w)}\varrho_\UD(z,w) \, .$$
	In terms of the hyperbolic distance $\hypdist_\UD$ instead of the pseudohyperbolic distance $\varrho_\UD$ this chain of inequalities is equivalent to
	$$ -\log \big(e^{-\hypdist_\UD(z,w)}+\delta \sinh \hypdist_\UD(z,w) \big)\le \hypdist_\UD \big(\varphi(z),\varphi(w) \big) \le \log \big(e^{\hypdist_\UD(z,w)}-\delta \sinh \hypdist_\UD(z,w) \big)\, ,$$
	or
	$$  -\log \big(1-\delta e^{-\hypdist_\UD(z,w)} \sinh \hypdist_\UD(z,w) \big) \le  \hypdist_\UD(z,w)-\hypdist_\UD \big(\varphi(z),\varphi(w) \big) \le \log \big(1+\delta e^{\hypdist_\UD(z,w)} \sinh \hypdist_\UD(z,w)\big)  \, .$$
	Taking into account that $\log(1+x) \le x \le -\log(1-x)$ for all $x \in [0,1)$,
	the proof is complete.
\end{proof}

If $\varphi\in\Hol(\UD,\UD)$ is not a conformal automorphism, then Beardon and Minda \cite[Corollary 3.7]{BM2004}  showed that \eqref{eq:sp-hypdifquot} leads to a Schwarz\,--\,Pick lemma for
the hyperbolic distortion $\dhyp \varphi$  of $\varphi$; namely
\begin{equation}\label{eq:Golusin_Estimates}
	\hypdist_{\UD}\big(\dhyp \varphi(z), \dhyp \varphi(w)\big) \leq 2 \hypdist_{\UD}(z,w) \, , \quad z ,w \in \UD.
\end{equation}
For earlier versions of this inequality, we refer to \cite[Theorem 3]{Golusin1945}, \cite[p.~335]{Go} and \cite{Yamashita1994,Yamashita1997}.
For our purpose the following immediate consequence of \eqref{eq:Golusin_Estimates} will be important
\begin{equation} \label{eq:Golusin2}
	e^{-2 \hypdist_\UD(z,w)} \le \frac{1-\dhyp \varphi(z)}{1-\dhyp \varphi(w)} \le e^{2 \hypdist_\D(z,w)} \, \quad z,w \in \UD \, .
\end{equation}
These inequalities are derived by writing \eqref{eq:Golusin_Estimates} in terms of the pseudo-hyperbolic distance and using the triangle inequality for $\varrho_{\UD}(\dhyp \varphi(z), \dhyp \varphi(w))$.

Let us point out that \eqref{eq:Golusin2} immediately yields a Lindel\"of--type angular limit theorem for the hyperbolic distortion $\dhyp \varphi$ at a point $\sigma \in \partial \D$:
\begin{equation} \label{EQ:RadialToSectorial}
	\lim_{r \to 1-} \dhyp \varphi(r\sigma )=1 \quad \text{ implies } \quad \anglim_{z \to \sigma} \dhyp \varphi(z)=1\, .
\end{equation}
This follows at once from the basic fact (see \cite[Section 2.2]{Marco}) that every sequence $(z_n)$ in $\UD$ converging non-tangentially to $\sigma$ has bounded hyperbolic distance to the segment $[0,\sigma)=\{r \sigma \, : \, 0 \le r <1\}$.
In fact, a statement much stronger than \eqref{EQ:RadialToSectorial} holds: it suffices to assume $\dhyp \varphi(z_n) \to 1$ for~a single sequence $(z_n)$ of bounded step converging non-tangentially to $\sigma$.

\begin{proposition}[Lindel\"of's theorem for hyperbolic distortion] \label{Prop:HypDer}
Let $\varphi \in \Hol(\UD,\UD)$ and $\sigma \in \partial \D$. Suppose that $(z_n)$ is a sequence in $\D$ of bounded step converging non-tangentially to $\sigma$ and such that
$\dhyp \varphi(z_n) \to 1$. Then $\dhyp \varphi(z) \to 1$ as $z \to \sigma$ non-tangentially.
\end{proposition}

The corresponding result for \textit{holomorphic} functions $\varphi : \UD\to \UD$ instead of the \textit{non-holomorphic} function $\dhyp \varphi : \D \to [0,1]$ is a special case of \cite[Theorem 5.6]{AB2016}. The most suitable setting for the proof of Proposition \ref{Prop:HypDer} is working on the right half-plane instead of the unit disk. Hence, we give the proof in Section~\ref{SUB:half-plane}, after a short introduction to hyperbolic distortion in the right half-plane model of hyperbolic geometry.

In the remainder of this subsection, we observe a necessary condition for the existence of the non-tangential limit for the hyperbolic distortion at a boundary point as well as an equivalence between continuous and discrete convergence for the hyperbolic distortion.

\begin{proposition} \label{prop:4}
Suppose $\varphi \in \Hol(\UD,\UD)$, $\sigma \in \partial \D$, and
$$ \int \limits_{0}^1 \big( 1-\dhyp \varphi(r \sigma) \big) \lambda_{\UD}(r \sigma) \, \di r <+\infty \, .$$
Then
$$ \angle \lim \limits_{z \to \sigma} \dhyp \varphi(z)=1 \, .$$
\end{proposition}
\begin{proof}
It suffices to consider the case $\sigma=1$. Fix $R \in (0,1)$.  Consider the automorphism of~$\UD$ given by $\phi_R(z):=\frac{z-R}{1-R z}$, ${z \in \UD}$. For $r\in[R,1)$ set $x:=\phi_R(r)$.  Using~\eqref{eq:Golusin2} and taking into account that ${\di x/\di r=\phi_R'(r)>0}$ for all ${r\in[R,1})$, we obtain
\begin{eqnarray*}
	\int \limits_{R}^1 \big( 1-\dhyp \varphi(r) \big) \lambda_{\UD}(r) \, \di r & \ge &
	\big( 1-\dhyp \varphi(R)\big)  \int \limits_{R}^{1} e^{-2 \hypdist_{\UD} (R,r)} \lambda_{\UD}(r) \, \di r\\
	&=&  \big(1-\dhyp \varphi(R)\big)  \int \limits_{R}^{1} e^{-2 \hypdist_{\UD} (0,\phi_R(r))} \lambda_{\D}(\phi_R(r)) |\phi'_R(r)| \, \di r \\
	&=& \big(1-\dhyp \varphi(R)\big)  \int \limits_{0}^{1} e^{-2 \hypdist_{\UD} (0,x)} \lambda_{\D}(x)  \, \di x\\
	&=&\big(1-\dhyp \varphi(R)\big)   \int \limits_{0}^{1} \frac{2(1-x)}{(1+x)^3}  \, \di x =\frac{1-\dhyp \varphi(R)}{2} \, .
\end{eqnarray*}
This proves that $\dhyp \varphi(z) \to 1$ as $z \to 1$ radially. Proposition~\ref{Prop:HypDer} then yields $\dhyp \varphi(z) \to 1$ as $z \to 1$ non-tangentially.
\end{proof}

\begin{proposition}\label{prop:5}
Let $\varphi\in\Hol(\UD,\UD)$ and let $\psi \in \Aut(\D)$ be given by $\psi(z):={(z+b)/(1+bz)}$ for all ${z\in\UD}$, with $b\in(0,1)$.  Let
$$ I:= \int \limits_{0}^1 \left(1-\dhyp \varphi(r) \right) \lambda_{\D}(r) \, dr \, ,\qquad S:= \sum \limits_{n=0}^{\infty} \big(1-\dhyp \varphi(\psi^{\circ n}(0) \big) \, .$$
Then
$$         \frac{2b}{\left(1+b\right)^2} \le \frac{\,I}{\,S\,}  \le  \frac{2b}{\left(1-b\right)^2}  \, .$$
\end{proposition}

\begin{proof}
In the same manner as in the proof of Proposition~\ref{prop:4}, we integrate from $\psi^{\circ n}(0)$ to $\psi^{\circ (n+1)}(0)$ and then summation leads to the desired inequalities.
\end{proof}

\addtocontents{toc}{\SkipTocEntry}\subsection{Passing to the right half-plane and proof of Proposition \ref{Prop:HypDer}.}\label{SUB:half-plane}\mbox{~}
By the Riemann mapping theorem and conformal invariance of the hyperbolic distortion, the results of Subsection \ref{SUB:hypdistortion} hold -- mutatis mutandis -- for holomorphic self-maps of any simply connected hyperbolic domain instead~of~$\D$.

In fact, most of the proofs in this paper are carried out in the half-plane model of hyperbolic geometry. For this reason, we make repeated use of the passage between the unit disk~$\UD$ and the right half-plane $\UH:=\{\zeta\in \C \, : \, \Re \zeta>0\}$ as described below.

Conjugation by the M\"obius map $T:\UH\to\UD$ given by $z={T(\zeta):=(\zeta-1)/(\zeta+1)}$, ${\zeta\in\UH}$,  transforms holomorphic self-maps~$\varphi$ of the unit disk~$\UD$ to holomorphic self-maps~$F$ of the right half-plane~$\UH$:
$$
F:\UH\to\UH,\quad F:=T^{-1}\circ \varphi\circ T.
$$
This correspondence is clearly one-to-one, which allows us to extend the notion of a conformality at a boundary point to self-maps of the half-plane~$\UH$. Precomposing a self-map ${F\in\Hol(\UH,\UH)}$ with an automorphism of~$\UH$, we can move a given boundary point to~$\infty$.
\begin{definition}
A holomorphic self-map $F\in\Hol(\UH,\UH)$ is said to be \emph{conformal at $\infty$ in the weak} (resp. \emph{strong}) \emph{sense}, if the self-map $\varphi:= {T\circ F\circ T^{-1}}$ of~$\UD$, where $T$ is defined as above, is conformal at~${\sigma:=1}$ in the weak (resp. strong) sense.
\end{definition}
In fact, it is not difficult to characterize conformality of~$F\in\Hol(\UH,\UH)$ at~$\infty$ directly, i.e. without passing to~$\UD$. First of all, both types of conformality imply that the angular limit
$\anglim_{\zeta\to\infty}F(\zeta)=:F(\infty)$  exists, finite or infinite, and belongs to~$\partial\UH$.
In such a case, we may in fact suppose that $\infty$~is a boundary fixed point of~$F$, i.e. ${\anglim_{\zeta\to\infty}F(\zeta)=\infty}$; otherwise we replace $F$ by ${1/(F-F(\infty))}$.
A self-map $F\in\Hol(\UH,\UH)$ satisfying the latter condition is conformal at~$\infty$ in the \textit{weak} sense if and only if the angular limit $\anglim_{\zeta\to\infty}\arg\big(F(\zeta)/\zeta\big)$ exists and equals~$0$. Note that the words ``and equals~$0$'' can be actually omitted (see Remark~\ref{RM_without-limit-values}).

Moreover, $F\in\Hol(\UH,\UH)$ is conformal at~$\infty$ in the \textit{strong} sense and has, at the same time, a boundary fixed point at~$\infty$ if and only if $F$ has the non-zero angular derivative at~$\infty$ in the sense of Carath\'eodory, i.e.
\begin{equation}\label{EQ_Carath-derivative}
F'(\infty):=\anglim_{\zeta\to\infty}\frac{F(\zeta)}{\zeta}\,>\,0.
\end{equation}
It is worth mentioning that the angular limit in~\eqref{EQ_Carath-derivative} exists for any holomorphic self-map of~$\UH$; it is always finite, and to be more precise, it is a non-negative real number, see e.g. \cite[\S26]{Valiron:book}.
Note also that for the corresponding self-map $\varphi\in\Hol(\UD,\UD)$, we have
$$
\anglim_{z\to1}\frac{\varphi(z)-1}{z-1}\,=\,\frac{1}{F'(\infty)},
$$
with the expression in the l.h.s. being exactly the angular derivative $\varphi'(1)$ provided that $\varphi$ has a boundary fixed point at~$1$.

Let us now recall the basics of  the hyperbolic geometry in the half-plane setting. The hyperbolic metric  $\lambda_\UH(\zeta) \, |\!\di\zeta|$ in $\UH$ is defined by
$$
\lambda_\UH(\zeta):=\lambda_{\UD}\big(T(\zeta)\big)\big|T'(\zeta)\big|=\frac{1}{\Re \zeta},\qquad \zeta\in\UH.
$$
The induced hyperbolic distance denoted by $\hypdist_\UH$ satisfies
$$
\hypdist_\UH(\omega,\zeta)=\hypdist_\UD\big(T(\omega),T(\zeta)\big)
$$
for any $\zeta,\omega\in\UH$.
Similarly to the unit disk setting, for ${F \in \Hol(\UH,\UH)}$, the \emph{hyperbolic distortion} $\dhyp F(\zeta)$ at~a point ${\zeta \in \UH}$ is defined by
$$ \dhyp F(\zeta):=\lim \limits_{\omega \to \zeta} \frac{\hypdist_\UH\big(F(\omega),F(\zeta)\big)}{\hypdist_\UH(\omega,\zeta)} =\frac{\lambda_\UH(F(\zeta))}{\lambda_\UH(\zeta)}\,|F'(\zeta)|=\frac{\Re \zeta}{\Re F(\zeta)} \, |F'(\zeta)|\,.
$$
It is not difficult to see that
\begin{equation*}\label{EQ_conf-inv-of-D_h}
\dhyp F(\zeta)=\dhyp \varphi(z)\quad\text{~if $~\varphi=T\circ F\circ T^{-1}$, $F\in\Hol(\UH,\UH)$, and $z=T(\zeta)$ with $\zeta\in\UH$.}
\end{equation*}

Let us now turn again to Proposition \ref{Prop:HypDer}. Recall that a sequence ${(\zeta_n)\subseteq\UH}$ is said to converge to~$\infty$ \emph{non-tangentially} if ${\lim_{n\to+\infty}\zeta_n=\infty}$ and there exists ${\theta\in(0,\pi/2)}$ such that
$$
(\zeta_n)\subseteq A_\theta:=\{\zeta \in \C \, : \, |\arg \zeta|<\theta\}.
$$

\begin{proof}[\proofof{Proposition \ref{Prop:HypDer}.}]
Passing to the right half-plane~$\UH$, as described above, consider a self-map ${F \in \Hol(\UH,\UH)}$ such that $\dhyp F(\zeta_n) \to 1$ for some sequence $(\zeta_n)$ in~$\UH$ converging non-tangentially to~$\infty$ and satisfying $\hypdist_{\UH}(\zeta_n,\zeta_{n+1}) \le M$ for all ${n\in\N}$ and some fixed ${M>0}$. Denote ${x_n:=|\zeta_n|}$. Since $(\zeta_n)$ belongs to some fixed sector $A_{\theta}$, ${\theta\in(0,\pi/2)}$, there is a constant ${L>0}$ such that ${\hypdist_{\UH}\big(\zeta_n,x_n\big) \le L}$ for all ${n \in \N}$, see \cite[Lemma 5.4.1 (6)]{BCD-Book}. It follows from~\eqref{eq:Golusin2} that
$$1-\dhyp F(x_n)  \le \big(1-\dhyp F(\zeta_n) \big) e^{2 \hypdist_{\UH}(x_n,\zeta_n)} \to 0 $$
as ${n \to +\infty}$. Moreover, we have
$\hypdist_{\UH}(x_n,x_{n+1}) \le \hypdist_{\UH}({\zeta_n,\zeta_{n+1}})\le M$, see e.g. \cite[Lemma~5.4.1~(5)]{BCD-Book}.
For ${n\in\Natural}$, denote by ${I_n\subseteq\Real}$ the segment with the end-points $x_{n}$ and~$x_{n+1}$. It is easy to see that the union $\cup_{n\in\N}\,I_n$ covers ${[x_1,+\infty)}$. Hence, for any sequence ${(t_k)\subseteq[x_1,+\infty)}$ there exists a sequence ${(n_k)\subseteq\N}$ such that ${t_k\in I_{n_k}}$ for each ${k\in\N}$. Clearly, if ${t_k\to+\infty}$ as ${k\to+\infty}$, then also ${n_k\to+\infty}$ as ${k\to+\infty}$. Note that $I_{n_k}$ is the hyperbolic geodesic segment joining $x_{n_k}$ with~$x_{n_k+1}$. Therefore,
$$\hypdist_\UH\big(t_k,x_{n_k}\big) \le \hypdist_\UH\big(x_{n_k},x_{n_k+1}\big) \le M \, $$
for all $k \in \N$. Consequently, in view of \eqref{eq:Golusin2},
$$ 1-\dhyp F(t_k) \le \big(1-\dhyp F(x_{n_k}) \big)\, \exp\!\big(2 \hypdist_{\UH}(t_k,x_{n_k})\big) \le \big(1-\dhyp F(x_{n_k}) \big) e^{2 M}~ \rightarrow 0~ \quad \text{as~$~k\to+\infty$}.$$
This proves that $\dhyp F(t) \to 1$ as $(0,+\infty)\ni  t \to +\infty$. Passing back to $\D$, we get that ${\dhyp \varphi(r\sigma) \to 1}$ as ${r \to 1^-}$, and then~\eqref{EQ:RadialToSectorial}
shows  that ${\dhyp \varphi(z) \to 1}$ as ${z \to \sigma}$ non-tangentially.
\end{proof}

\section{Proof of Theorem~\ref{TH-main1}} \label{SEC:ProofTH1}

The proof of Theorem~\ref{TH-main1} is long and will therefore be divided into several steps. We begin with an auxiliary result about holomorphic self-maps of $\UH$ which have a boundary fixed point at $\infty$ and are conformal at $\infty$ in the weak sense. This result will also be needed for the proof of Theorem \ref{TH-main2}.

\begin{proposition}\label{PR_Real_axis_in_image}
Suppose that $F\in\Hol(\UH,\UH)$ has a boundary fixed point at~$\infty$, i.e. $$\anglim_{\zeta\to\infty}F(\zeta)=\infty.$$If $F$ is conformal at~$\infty$ in the weak sense, then there exists a smooth injective curve ${\gamma:[0, +\infty) \to \UH}$ satisfying the following conditions:
\begin{romlist}
	\item\label{IT_gamma_arc-len-param}  $\gamma$ is parameterized by Euclidean arc length, i.e. ${|\gamma{\,}'(s)|=1}$ for all ${s\ge0}$;
	\item\label{IT_gamma_trace} the trace of $F\circ\gamma$ coincides with the half-line ${[R,+\infty)}$ for some~${R>0}$;
		\item\label{IT_gamma_Re} $\Re(\gamma(s))/s\to1$ as $~{s\to+\infty}$;
			\item\label{IT_gamma_Im} $\arg(\gamma(s))\to 0~$ as $~{s\to +\infty}$;
			\item\label{IT_gamma-bdd-hyperbolic-distance} there is a constant $M>0$ such that $\hypdist_{\UH}\big(\gamma(s),s\big) \le M$ for all ${s \ge 1}$.
		\end{romlist}
\begin{proof}
			This is essentially a corollary of Proposition~\ref{PR_known1} proved in the Appendix.
			Indeed, by Proposition~\ref{PR_known1}, there exists a domain $U\subseteq\UH$ on which $F$ is univalent and such that $W:=F(U)$ is isogonal at~$\infty$, see Definition~\ref{DF_domain-isogonality}. This allows us to define, choosing a suitably large constant ${R>0}$, an injective smooth curve  $\gamma_0:[0,+\infty)\to\UH$ by setting ${\gamma_0(x):=G(x+R)}$ for all~${x\ge0}$, where
			${G:W\to U}$ is the inverse of~$F|_U$. Note that the length of~$\gamma_0$ is infinite because by Proposition~\ref{PR_known1}\,\hbox{\ref{IT_known1-F_inverse-isogonal}}, ${G(x)\to\infty}$ as ${(0,+\infty)\ni x\to+\infty}$. Now denote by~$\gamma$ the re-parametrization of~$\gamma_0$ by the euclidean arc-length parameter, i.e.
			$$
			\gamma:[0,+\infty)\to\UH,\quad \gamma(s):=\gamma_0\big(\ell^{-1}(s)\big),~s\ge0,\quad \text{where~}~
			\ell(x):=\int\nolimits_0^x |G'(y)|\di y,~\text{$~x\ge0$}.
			$$
			Properties~\ref{IT_gamma_arc-len-param} and~\ref{IT_gamma_trace} of~$\gamma$ hold by construction.
			Property~\ref{IT_gamma_Im} holds because $\arg G(x)\to0$ as ${(0,+\infty)\ni x\to+\infty}$ by Proposition~\ref{PR_known1}\,\ref{IT_known1-F_inverse-isogonal}. Furthermore, in view of the same assertion in Proposition~\ref{PR_known1}, we have
			$$
			\Re\gamma'(s)=\frac{\Re G'\big(\ell^{-1}(s)\big)}{|G'\big(\ell^{-1}(s)\big)|}\,\to\,1\quad\text{as~$~s\to+\infty$},
			$$
			which implies property~\ref{IT_gamma_Re}. Finally, the last property~\ref{IT_gamma-bdd-hyperbolic-distance} follows from the previous two.
		\end{proof}
	\end{proposition}

	We now turn to the proof that all the conditions \ref{IT_TH-main1:hder-to-1}, \ref{IT_Th-main1:hdiffquo-to-1} and \ref{IT_TH-main1:semireg-cntct-pnt} listed in Theorem~\ref{TH-main1} and conditions \ref{IT_FEQ_a}, \ref{IT_FEQ_c} and  \ref{IT_FEQ_b} listed in Addendum  \ref{ADD:TH-main1} are pairwise equivalent. As for the equivalence of \ref{IT_TH-main1:semireg-cntct-pnt} and \ref{IT_FEQ_b} see Remark \ref{RM_Yamashita}.
	We structure the proof of the remaining implications as follows:
		
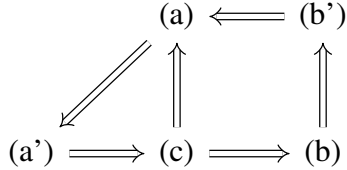
\begin{center}\begin{tikzcd}[row sep=large]

			\quad & \arrow[Rightarrow,ld]  \text{\ref{IT_TH-main1:hder-to-1}} & \arrow[Rightarrow,l]  \text{\ref{IT_FEQ_c}} \\
			\text{\ref{IT_FEQ_a}}  \arrow[Rightarrow,r]   &
			\text{\ref{IT_TH-main1:semireg-cntct-pnt}}\arrow[Rightarrow,r]
			\arrow[Rightarrow,u]   &
			\text{\ref{IT_Th-main1:hdiffquo-to-1}}\arrow[Rightarrow,u]
		\end{tikzcd}
	\end{center}

	The implication \ref{IT_Th-main1:hdiffquo-to-1}~$\Longrightarrow$~\ref{IT_FEQ_c} is of course trivial, and \ref{IT_TH-main1:hder-to-1}~$\Longrightarrow$~\ref{IT_FEQ_a} holds in view of~\eqref{EQ:RadialToSectorial}. The implications \ref{IT_FEQ_c}~$\Longrightarrow$~\ref{IT_TH-main1:hder-to-1} and  \ref{IT_TH-main1:semireg-cntct-pnt}~$\Longrightarrow$~\ref{IT_TH-main1:hder-to-1} are not completely obvious, but not difficult to establish. The major chain of implications is \ref{IT_FEQ_a}~$\Longrightarrow$ \ref{IT_TH-main1:semireg-cntct-pnt} $\Longrightarrow$~\ref{IT_Th-main1:hdiffquo-to-1}.

	It is also worth mentioning that, from a formal point of view, we do not have to prove the implication \ref{IT_TH-main1:semireg-cntct-pnt}~$\Longrightarrow$~\ref{IT_TH-main1:hder-to-1}. Nevertheless, we will give a direct proof of this implication as it turns out to be very useful in the proof of \ref{IT_TH-main1:semireg-cntct-pnt}~$\Longrightarrow$~\ref{IT_Th-main1:hdiffquo-to-1}.

	\begin{proof}[\proofof{Theorem~\ref{TH-main1}: \ref{IT_FEQ_c} $\Longrightarrow$ \ref{IT_TH-main1:hder-to-1}}]
		Let $(z_n)$ be a sequence in~$\D$ satisfying condition~\ref{IT_FEQ_c} in Addendum~\ref{ADD:TH-main1}. Denoting \begin{gather*}
			d_n:=\hypdist_{\UD}\big(z_n,z_{n+1}\big),\qquad q_n:=\hypdist_\UD\big(\varphi(z_n),\varphi(z_{n+1})\big),\\
			\alpha_n:=\dhyp\varphi(z_n)~\text{\,if~${d_n=0}$}\quad\text{and}\quad \alpha_n:=q_n/d_n~\text{\,otherwise},
		\end{gather*}
		this means that
		$d_n \le C:=\sup_{k \in \N} d_k<+\infty$ and $\alpha_n\to1$ as ${n\to+\infty}$. Proposition~\ref{prop:Ineq_ApprUniv} implies that if $d_n\neq0$, then
		$$
		1-\dhyp \varphi(z_n) \le e^{d_n} \frac{d_n}{\sinh \big(d_n\big)} \left( 1-\alpha_n \right) \, .
		$$
		Since $x \mapsto e^{x} x/\sinh(x)$ is a strictly increasing function on $(0,+\infty)$, we  get
		$$
		0 \le 1-\dhyp \varphi(z_n) \le \frac{C e^{C}}{\sinh \big(C \big)} \left( 1-\alpha_n \right) \, ,
		$$
		which is trivially true also if~${d_n=0}$. It follows that ${\dhyp \varphi(z_n) \to 1}$ as ${n \to \infty}$. Therefore, Proposition~\ref{Prop:HypDer} shows that ${\dhyp \varphi(z) \to 1}$ as ${z \to \sigma}$ non-tangentially, and~\ref{IT_TH-main1:hder-to-1} follows immediately.
	\end{proof}

	\begin{proof}[\proofof{Theorem~\ref{TH-main1}: \ref{IT_TH-main1:semireg-cntct-pnt} $\Longrightarrow$ \ref{IT_TH-main1:hder-to-1} and \ref{IT_Th-main1:hdiffquo-to-1}}]

		Without loss of generality we may assume that ${\sigma=1}$. Passing to the right-half plane as explained in Section~\ref{SUB:half-plane}, we consider a self-map $F\in\Hol(\UH,\UH)$ with a  boundary fixed point at~$\infty$ and such that
		$$ \anglim_{\zeta\to \infty} \arg \frac{F(\zeta)}{\zeta} =0 \, .$$

		\StepG{Step 1:} \ref{IT_TH-main1:semireg-cntct-pnt} $\Longrightarrow$ \ref{IT_TH-main1:hder-to-1}.  Lemma~\ref{LM_isogonality} shows that \eqref{EQ_LM_isogonality} holds, that is,
$$ 
\frac{\zeta F'(\zeta)}{F(\zeta)} \to 1 \quad \text{ as }~\zeta \to \infty~\text{ non-tangentially} \, .
$$
		But then
		$$ 1 \ge \dhyp F(x)=\frac{x |F'(x)|}{\Re F(x)}
		\ge \left| \frac{x F'(x)}{F(x)} \right| \to 1 \quad \text{ as }~(0,+\infty)\ni x\to +\infty .$$
		Hence, $\dhyp F(x) \to 1$ as $(0,+\infty)\ni x \to +\infty$ and, as a result, the half-plane version of Proposition~\ref{Prop:HypDer} implies
		\begin{equation}\label{EQ_step1}
			\anglim_{\zeta\to\infty} \dhyp F(\zeta)=1,
		\end{equation}
		as desired.

		\StepG{Step 2:} \ref{IT_TH-main1:semireg-cntct-pnt} $\Longrightarrow$ \ref{IT_Th-main1:hdiffquo-to-1}.
		We have to prove the half-plane version of~\ref{IT_Th-main1:hdiffquo-to-1}, i.e.
		\begin{equation}\label{EQ_half-plane-version-of-c}
			\anglim \limits_{\omega,\zeta \to \infty} \frac{\hypdist_\UH\big(F(\omega),F(\zeta)\big)}{\hypdist_\UH\big(\omega,\zeta\big)}=1\,.
		\end{equation}

		\StepG{Step 2A}. We first show that
		\begin{equation}\label{EQ:HypDistQuotientAlongGamma}
			\lim \limits_{t>s\to+\infty} \frac{\hypdist_{\UH}\big(F(\gamma(s)),F(\gamma(t))\big)}{\hypdist_{\UH}\big(\gamma(s),\gamma(t)\big)} =1 \, ,
		\end{equation}
		where $\gamma:[0,+\infty)\to\UH$ is the smooth injective curve constructed in Proposition~\ref{PR_Real_axis_in_image}.

		Denote by ${[a,b]_h}$ the hyperbolic geodesic segment in $(\UH,\hypdist_\UH)$ joining $a$ to~$b$ and observe that by construction $G(s,t):={\big[F(\gamma(s)),F(\gamma(t))\big]_h}=F\big(\gamma|_{[s,t]}\big)$. Therefore,
		\begin{multline*}
			\hypdist_\UH\big(F(\gamma(s)),F(\gamma(t))\big) = \int \limits_{G(s,t)} \lambda_\UH(\omega) \, |\!\di{\omega}|~=\int \limits_{\gamma|_{[s,t]}} \dhyp F(\zeta) \, \lambda_\UH(\zeta) \, |\!\di{\zeta}|\\
			\ge\int \limits_{\gamma|_{[s,t]}} \lambda_\UH(\zeta) \, |\!\di{\zeta}|~\cdot~\min \limits_{\zeta \in \mathrm{tr}(\gamma|_{[s,t]})} \dhyp F(\zeta)
			~\ge~\hypdist_\UH\big(\gamma(s),\gamma(t)\big) ~\cdot~\min\limits_{\zeta \in \mathrm{tr}(\gamma|_{[s,t]})} \dhyp F(\zeta).
		\end{multline*}
		Moreover, since $\gamma(s)\to\infty$ non-tangentially as ${s\to+\infty}$ and since \eqref{EQ_step1} holds by \textsc{Step~1}, we have ${\dhyp F(\gamma(s)) \to 1}$ as ${s \to +\infty}$. Combined with the above inequality, this implies~\eqref{EQ:HypDistQuotientAlongGamma}.

		\StepG{Step 2B}. Now we will complete the proof by deducing~\eqref{EQ_half-plane-version-of-c} based on~\eqref{EQ:HypDistQuotientAlongGamma}.
		Suppose~\eqref{EQ_half-plane-version-of-c} does not hold. Then there are sequences $(\omega_n)$ and $(\zeta_n)$ in~$\UH$ converging non-tangentially to~$\infty$ such that
		\begin{equation} \label{EQ:HypDistQuotientAlongGammaContradictionP}
			\lim \limits_{n \to +\infty} \frac{\hypdist_\UH\big(F(\omega_n),F(\zeta_n) \big)}{\hypdist_\UH \big( \omega_n,\zeta_n \big)}<1 \, .
		\end{equation}
		Removing a finite number of terms, we may suppose that ${s_n:=|\omega_n|\ge1}$ and ${t_n:=|\zeta_n|\ge1}$ for all ${n\in\Natural}$. 
                Passing if necessary to subsequences, we may assume that either ${\hypdist_\UH(\omega_n,\zeta_n)\to+\infty}$ as ${n\to\infty}$, or $\hypdist_\UH(\omega_n,\zeta_n)\le M$ for some constant~${M>0}$ and all~${n\in\Natural}$. In the latter case,  we note that in \textsc{Step~1} we have already proved that condition \ref{IT_TH-main1:semireg-cntct-pnt} implies $\dhyp F(\omega) \to 1$ as $\omega \to \infty$ non-tangentially. In view of (the half-plane version of) Proposition~\ref{prop:Ineq_ApprUniv}, we therefore have
		\begin{align*}
			0 & \le 1-\frac{\hypdist_\UH\big(F(\omega_n),F(\zeta_n)\big)}{\hypdist_\UH(\omega_n,\zeta_n)} =\frac{\hypdist_\UH(\omega_n,\zeta_n)-\hypdist_\UH\big(F(\omega_n),F(\zeta_n)\big)}{\hypdist_\UH(\omega_n,\zeta_n)} \\
			&\le e^{\hypdist_\UH(\omega_n,\zeta_n)} \frac{\sinh \big(\hypdist_\UH(\omega_n,\zeta_n)\big)}{\hypdist_\UH(\omega_n,\zeta_n)} \big(1- \dhyp F(\omega_n) \big) \le e^M\,\frac{\sinh M}{M} \big(1-\dhyp F(\omega_n) \big) \to 0
		\end{align*}
		as $n\to\infty$, a contradiction to \eqref{EQ:HypDistQuotientAlongGammaContradictionP}.

		So we may suppose that  ${\hypdist_\UH(\omega_n,\zeta_n)}$ tends to~$+\infty$. By Property~\ref{IT_gamma-bdd-hyperbolic-distance} in Proposition~\ref{PR_Real_axis_in_image}, the distances $\hypdist_\UH\big(s_n,\gamma(s_n)\big)$  and $\hypdist_\UH\big(t_n,\gamma(t_n)\big)$ are bounded.
		Moreover, since $(\omega_n)$ and $(\zeta_n)$ are contained in some sector $A_\theta\subseteq\UH$, ${\theta\in(0,\pi/2)}$, also the distances $\hypdist_{\UH}\big(\omega_n,s_n\big)$ and $\hypdist_{\UH}\big(\zeta_n,t_n\big)$ are bounded. Using the triangle inequality and the Schwarz\,--\,Pick Lemma, we therefore conclude that
		\begin{equation}\label{EQ_hypdist_bdd}
			\begin{array}{ll}
				\hypdist_\UH\big(F(\omega_n),F(\gamma(s_n))\big)\le\hypdist_\UH\big(\omega_n,\gamma(s_n)\big)\le C_0 &\text{~and~}\\[1ex] \hypdist_\UH\big(F(\zeta_n),F(\gamma(t_n))\big)\le\hypdist_\UH\big(\zeta_n,\gamma(t_n)\big)\le C_0 &\text{~for all ${n\in\Natural}$}
			\end{array}
		\end{equation}
		and some constant ${C_0>0}$ large enough. Now, using the triangle inequality and~\eqref{EQ_hypdist_bdd}, we obtain
		\begin{align} \label{EQ:HypDistQuotientAlongGammaP1}
			\frac{\hypdist_\UH \big(\gamma(s_n),\gamma(t_n)\big)}{\hypdist_\UH(\omega_n,\zeta_n)} \ge \frac{\hypdist_\UH(\omega_n,\zeta_n)-\hypdist_\UH\big(\omega_n,\gamma(s_n)\big)-\hypdist_\UH\big(\zeta_n,\gamma(t_n)\big)}%
			{\hypdist_\UH(\omega_n,\zeta_n)}=:A_n \to 1 \quad \text{as~}~ n \to \infty.
		\end{align}
		In particular,  it follows that $\hypdist_\UH \big(\gamma(s_n),\gamma(t_n)\big)\to+\infty$ as ${n\to\infty}$. Thanks to~\eqref{EQ:HypDistQuotientAlongGamma}, it further follows that $\hypdist_\UH \big(F(\gamma(s_n)),F(\gamma(t_n))\big)\to+\infty$.
		Taking into account that by the triangle inequality,
		$$
		\hypdist_\UH\big(F(\gamma(s_n)),F(\gamma(t_n))\big) \le \hypdist_\UH\big(F(\omega_n),F(\zeta_n)\big)+\hypdist_\UH\big(F(\omega_n),F(\gamma(s_n))\big)+
		\hypdist_\UH\big(F(\zeta_n),F(\gamma(t_n))\big),
		$$
		and using again~\eqref{EQ_hypdist_bdd}, we conclude that also $\hypdist_\UH\big(F(\omega_n),F(\zeta_n)\big)\to+\infty$ as ${n\to\infty}$. Thus,
		\begin{multline} \label{EQ:HypDistQuotientAlongGammaP2}
			\frac{\hypdist_\UH\big(F(\gamma(s_n)),F(\gamma(t_n))\big)}{\hypdist_\UH\big(F(\omega_n),F(\zeta_n)\big)}\le B_n:=  \\[1.25ex]\text{~}\qquad= \frac{\hypdist_\UH\big(F(\omega_n),F(\zeta_n)\big)+\hypdist_\UH\big(F(\omega_n),F(\gamma(s_n))\big)+
				\hypdist_\UH\big(F(\zeta_n),F(\gamma(t_n))\big)}%
			{\hypdist_\UH\big(F(\omega_n),F(\zeta_n)\big)} \to 1  ~\text{~as~}~ n \to \infty.
		\end{multline}
		Combining~\eqref{EQ:HypDistQuotientAlongGamma}, \eqref{EQ:HypDistQuotientAlongGammaP1}, and~\eqref{EQ:HypDistQuotientAlongGammaP2}, we have
		\begin{eqnarray*}
			\frac{\hypdist_\UH\big(F(\omega_n),F(\zeta_n)\big)}{\hypdist_\UH \big(\omega_n,\zeta_n\big)} &=& \frac{~\hypdist_{\UH}\big(F(\gamma(s_n)),F(\gamma(t_n))\big)}{\hypdist_{\UH}\big(\gamma(s_n),\gamma(t_n)\big)~}\,%
			\frac{~\hypdist_\UH \big(\gamma(s_n),\gamma(t_n)\big)~}{\hypdist_\UH(\omega_n,\zeta_n)}\,%
			\frac{\hypdist_\UH\big(F(\omega_n),F(\zeta_n)\big)}{~\hypdist_\UH\big(F(\gamma(s_n)),F(\gamma(t_n))\big)~}\\
			&\ge&%
			\frac{~\hypdist_{\UH}\big(F(\gamma(s_n)),F(\gamma(t_n))\big)}{\hypdist_{\UH}\big(\gamma(s_n),\gamma(t_n)\big)~}\,%
			\frac{\,\,A_n\,}{B_n}~\,~\to~\,~1\quad \text{as~}~n\to\infty.
		\end{eqnarray*}
		This clearly contradicts~\eqref{EQ:HypDistQuotientAlongGammaContradictionP}, and thus the proof of~\eqref{EQ_half-plane-version-of-c}  is complete.
	\end{proof}

	In order to complete the proof of Theorem~\ref{TH-main1} we are left to prove that~\ref{IT_FEQ_a} implies~\ref{IT_TH-main1:semireg-cntct-pnt}. This is the tricky part. We begin with some preliminary observations.
	Since \ref{IT_FEQ_a} implies \ref{IT_TH-main1:hder-to-1} by Proposition~\ref{Prop:HypDer}, we can assume ${\dhyp \varphi(r\sigma) \to 1}$ as ${r \to 1}$. Clearly, without loss of generality we may also suppose that ${\sigma=1}$. Then, passing to the right-half plane as explained in Section~\ref{SUB:half-plane}, we see that the corresponding  holomorphic self-map~$F$ of~$\UH$ satisfies
	\begin{equation}\label{EQ_hderlim}
		\dhyp F(x)=\frac{x |F'(x)|}{\Re F(x)}\to 1\quad\text{as~$~x\to+\infty$.}
	\end{equation}
	Now fix some $x>0$, and denote
	\begin{equation*}\label{EQ_FxMapping}
		F_x \in \Hol(\UH,\UH), \, \quad F_x(\zeta):=\frac{F(x\zeta)-i\Im F(x)}{\Re F(x)},\quad \zeta\in\UH.
	\end{equation*}
	Note that $F_x(1)=1$. The idea now is to modify $F_x$ by precomposing it with a suitable automorphism of~$\UH$ to obtain a self-map $\widetilde{F}_x$ of~$\H$ with $\widetilde{F}_x(1)=1$ whose euclidean derivative $\widetilde{F}'_x(1)$ at the point $1$ coincides with the hyperbolic distortion $\dhyp F(x)$ at $x$.
	To this end, take $\mu \in \C^*=\C \setminus \{0\}$ and denote  by $L(\cdot;\mu)$ the unique automorphism of~$\UH$ satisfying ${L(1;\mu)=1}$ and ${\overline\mu L'(1;\mu)>0}$. Here the prime denotes the derivative w.r.t.~the first argument. For $\mu\not\in(0,+\infty)$ one can easily obtain the following explicit formula
	\begin{equation}\label{EQ_explicit_f-la}
		L(\zeta;\mu)=\frac{1+ia(\mu)\zeta}{ia(\mu)+\zeta},\quad a(\mu):=\frac{\Im\mu}{|\mu|-\Re\mu}\qquad (\zeta\in\UH).
	\end{equation}
	For $\mu>0$, we set $a(\mu):=\infty$, which results in $L(\cdot;\mu)=\id_\UH$. Clearly, ${a(\mu)=a(\rho\mu)}$ for any ${\rho>0}$.
	Moreover, it is not difficult to check that
	\begin{equation}\label{EQ_double}
		a\big(L'(2;\mu)\big)=\frac{a(\mu)}{2}
	\end{equation}
	whenever $\mu\neq0$. Finally, since ${\mu/|\mu|=(a+i)/(a-i)}={(a^2-1+2ia)/(a^2+1)}$, where ${a:=a(\mu)}$, we have
	\begin{equation}\label{EQ_product-formula}
		a(\mu_1\mu_2)=\frac{a(\mu_1)a(\mu_2)-1}{a(\mu_1)+a(\mu_2)}\quad\text{for all $~{\mu_1,\mu_2\in\C^*}$}.
	\end{equation}
	Alternatively, one can check that $a(\mu)=\cot(\frac12{\arg \mu})$. Then \eqref{EQ_product-formula} is the well-known addition formula for the cotangent function.
	Now, whenever $F'(x)\neq0$, we define
	$$\widetilde F_x:=F_x\circ L^{-1}(\cdot;F'(x))\, , $$
	and observe that
	\begin{equation}\label{EQ_arg-diff}
		F'(x\zeta)=\frac{\Re F(x)}{x}\,F_x'(\zeta)= \frac{\Re F(x)}{x}\,L'(\zeta;F'(x))\,\widetilde F'_x\big(L(\zeta;F'(x))\big),\quad \zeta\in\UH \, .
	\end{equation}
	In particular, $\widetilde{F}'_x(1)=\dhyp F(x)$.

	\begin{proof}[\proofof{Theorem~\ref{TH-main1}: \ref{IT_FEQ_a} $\Longrightarrow$ \ref{IT_TH-main1:semireg-cntct-pnt}}]

		In order to make use of our preliminary observations,   we pass to the right-half plane (assuming without loss of generality that ${\sigma=1}$) and consider a self-map $F\in\Hol(\UH,\UH)$ that satisfies~\eqref{EQ_hderlim}. Replacing, if necessary, $F(\zeta)$ by $F(\zeta+C)$ for a suitable ${C>0}$, we may further assume $F'(x)\neq0$ for all real ${x\ge1}$. Our goal is to prove that the angular limit ${F(\infty):=\anglim_{\zeta\to\infty}F(\zeta)}$ exists, with ${F(\infty)\in\partial\UH}$, and that the self-map $\widehat F$, obtained by post-composing~$F$ with an automorphism of~$\UH$ that moves $F(\infty)$ to~$\infty$,  satifies the condition ${\anglim_{\zeta\to\infty}\widehat F'(\zeta)/\big|\widehat F'(\zeta)\big|=1}$. Then we will use Lemma~\ref{LM_isogonality} to deduce that  $\widehat F$ is conformal at~$\infty$ in the weak sense, as desired.

		We divide the proof into several steps.

		\StepG{Step 1:} We first prove that ${\kappa_n:=F'(x_n)/|F'(x_n)|}$, where ${x_n:=2^n}$, converges. We will actually show that ${\kappa_n\to\pm1}$ as ${n\to+\infty}$. This statement is equivalent to
		\begin{subequations}
			\begin{align}
				\label{EQ_arg-to-zero}
				a\big(F'(x_n)\big)&\to\infty &&\text{as~$~n\to+\infty$}\\
				\label{EQ_arg-to-pi}
				\text{\llap{or\qquad}} a_n:=a\big(F'(x_n)\big)&\to0 &&\text{as~$~n\to+\infty$,}
			\end{align}
		\end{subequations}
		where $a:\C^*\to\Real\cup\{\infty\}$ is defined in~\eqref{EQ_explicit_f-la}.

		Suppose that~\eqref{EQ_arg-to-zero} does not hold. Then there exists a real number $b>0$ such that $|a_n|<b$ for infinitely many distinct values of~${n\in\Natural}$. We will prove that in this situation~\eqref{EQ_arg-to-pi} holds.

		Denote ${\widetilde a_n:=a\big(\widetilde F'_{x_n}\big(L(2;F'(x_n))\big)\big)}$. Note that
		$$
		L(2;F'(x_n))\in\BH(1,\log 2)\coloneqq\{\zeta\in\UH: \hypdist_\UH(\zeta,1)\leq \log2\}\quad\text{for all $~{n\in\Natural}$,}
		$$
		because $\hypdist_\UH(2,1)=\log 2$ and $L(\cdot;\mu)$ preserves the hyperbolic distance to~$1$.
		Therefore, taking into account that ${\widetilde F_{x_n}'(1)=\dhyp F(x_n)\to %0
			1}$ as ${n\to+\infty}$ and applying Lemma~\ref{LM_qualitative} to~$\widetilde F_{x_n}$, we see that
		\begin{equation}\label{EQ_atilde-to-infty}
			|\widetilde a_n|\to+\infty\quad\text{as~}~n\to+\infty.
		\end{equation}
		In particular, it follows that there exists ${n_0\in\Natural}$ such that ${|a_{n_0}|<b}$ as well as ${|\widetilde a_{n}|>(2+b^2)/b}$ for all ${n\ge n_0}$. Combining~\eqref{EQ_arg-diff} for~${\zeta:=2}$ and ${x:=x_n}$ with~\eqref{EQ_double} and~\eqref{EQ_product-formula}, we get
		\begin{equation}\label{EQ_a_n+1}
			a_{n+1}=%
			\frac{\tfrac12a_n \widetilde a_n-1}{\widetilde a_n+\tfrac12 a_n}\quad\text{for all~}~n\in\Natural.
		\end{equation}
		In particular, for ${n:=n_0
		}$, we have
		$$
		|a_{n_0+1}|
		\le\frac{1\,+\,\tfrac12b|\widetilde a_{n_0}|}{|\widetilde a_{n_0}|\,-\,\tfrac12b}%
		< b.
		$$
		For the second inequality we have used that as a function of $|\widetilde a_{n_0}|$ the rational expression above is~a decreasing function. Repeating the same argument with $n_0$ replaced by ${n_0+m}$ with ${m=1,2,\ldots}$, we see that
		\begin{equation*}\label{EQ_a-small-for-all}
			|a_n|<b\quad\text{for all~}~n\ge n_0.
		\end{equation*}
		Using again~\eqref{EQ_a_n+1} we can write $a_{n+1}=\tfrac12a_n+c_n$, where ${c_n\in\Real}$  satisfies
		$$
		|c_n|=%
		\left|\frac{1+\tfrac14a_n^2}{\widetilde a_n+\tfrac12 a_n}\right|%
		\le\frac{1+\tfrac14b^2}{|\widetilde a_n|-\tfrac12b}.
		$$
		for all $n \ge n_0$.
		In particular, in view of~\eqref{EQ_atilde-to-infty}, we have that  ${|c_n|\to0}$ as ${n\to+\infty}$.
		Thus,
		for any $m\in\Natural$,
		$$
		|a_{n_0+m}|=\left|\frac{a_{n_0}}{2^m}+\sum_{k=1}^{m}\frac{c_{n_0+k-1}}{2^{m-k}}\right|~\to~0\quad\text{as~}~m\to+\infty,
		$$
		which proves~\eqref{EQ_arg-to-pi}.

		\StepG{Step 2:} We next prove that
		\begin{subequations}
			\begin{align}
				\text{either}\quad \anglim_{\zeta \to \infty} F'(\zeta)/|F'(\zeta)|
				&= 1,\label{EQ+1}\\%
				\text{or}\quad \anglim_{\zeta \to \infty} \zeta^2 F'(\zeta)/\big|\zeta^2 F'(\zeta)\big|
				&= -1,\label{EQ-1}
			\end{align}
		\end{subequations}
		where the angular limits are to be understood in the sense of Definition~\ref{DF_anglim-in-iso-domain}.

		Suppose that $F'(2^n)/|F'(2^n)| \to 1$. Then \eqref{EQ+1} holds.
		Indeed, let
		$$
		G_n(\zeta):=F_{2^n}(\zeta)=\frac{F(2^n \zeta)-i \Im F(2^n)}{\Re F(2^n)} , \qquad \zeta \in \H \, .
		$$
		Then $G_n\in\Hol(\UH,\UH)$, $G_n(1)=1$, and
		$$
		G_n'(1)=\frac{2^n}{\Re F(2^n)} F'(2^n)=\dhyp F(2^n) \frac{F'(2^n)}{|F'(2^n)|} \to 1 \quad \text{ as } n \to \infty \, .
		$$
		This implies that every locally uniformly convergent subsequence of $(G_n)$ converges to the identity, so by a normal family argument $(G_n)$ converges locally uniformly in~$\H$ to the identity map. Hence
		$$
		\frac{F'(2^n \zeta)}{|F'(2^n \zeta)|}=\frac{G_n'(\zeta)}{|G_n'(\zeta)|} \to 1
		$$
		locally uniformly\footnote{Since $G_n'$ may vanish at some points of~$\UH$, the locally uniform convergence in this case should be understood in a slightly generalized sense, see Remark~\ref{RM_l.u.-convergence}.} in $\H$ as $n\to+\infty$. Now Lemma~\ref{lem:DiscreteHypDisksToAngularLimit} applied to $H:=F'/|F'|$ implies \eqref{EQ+1}.

		Next suppose $F'(2^n)/|F'(2^n)| \to -1$. Then the holomorphic functions $G_n : \H \to \H$ defined as above
		fix the point $1$ and $G_n'(1) \to -1$ as $n \to \infty$. This implies that $(G_n)$ converges locally uniformly in $\H$ to  ${\zeta \mapsto 1/\zeta}$. Applying Lemma~\ref{lem:DiscreteHypDisksToAngularLimit} to $H(\zeta):={\zeta^2F'(\zeta)/\big|\zeta^2F'(\zeta)\big|}$,  we see that in this case \eqref{EQ-1} holds.

		\StepG{Step 3.} In this step we complete the proof of \ref{IT_FEQ_a}~$\Longrightarrow$~\ref{IT_TH-main1:semireg-cntct-pnt} in case \eqref{EQ+1} holds.
		We first show
		\begin{equation}\label{EQ_to-infty}
			\lim \limits_{x \to +\infty} \Re~F(x)=+\infty.
		\end{equation}

		We have $\dhyp F(x) \to 1$  and $\kappa(x):={F'(x)/|F'(x)|}\to1$ as $x \to+\infty$. Therefore, by denoting ${f(x):=\log(\Re F(x))}$ for ${x\in(0,+\infty)}$, we obtain
		$$
		xf'(x)=\frac{x\Re F'(x)}{\Re F(x)}=\Re\kappa(x)\dhyp F(x)~\to~1\quad\text{as~$~(0,+\infty)\ni x \to+\infty$.}
		$$
		Hence, $\int\limits_1^{+\infty}f'(x)\,\di x=+\infty$, which easily implies~\eqref{EQ_to-infty}.

		By~\eqref{EQ_to-infty} we have ${F(x)\to\infty}$ as ${(0,+\infty)\ni x\to+\infty}$. Therefore, in view of Lindel\"of's angular limit theorem, see e.g. \cite[Theorem 1.5.7]{BCD-Book} or \cite[Theorem~9.3]{Pombook75}, it follows that ${F(\zeta)\to\infty}$ as ${\zeta\to\infty}$ non-tangentially; i.e., $F$ has a boundary fixed point at~$\infty$. Hence,~\eqref{EQ+1} implies, using Lemma~\ref{LM_isogonality}, that
		$$ \angle \lim \limits_{\zeta \to \infty} \arg \frac{F(\zeta)}{\zeta}=0 \, .$$
		Passing back to the unit disk, we see that $\varphi \in \Hol(\UD,\UD)$
		has the boundary fixed point~$1$ and
		$$ \angle \lim \limits_{z \to 1} \arg \frac{1-\varphi(z)}{1-z}= 0 \, .$$
		This completes the proof of the implication \ref{IT_FEQ_a}~$\Longrightarrow$~\ref{IT_TH-main1:semireg-cntct-pnt} in Theorem~\ref{TH-main1} under the assumption that~\eqref{EQ+1} holds.

		\StepG{Step 4.}  We are left to consider the alternative~\eqref{EQ-1}.

		Essentially the same argument as in \textsc{Step~3} allows us to conclude that in this case, $f(x):=\log (\Re F(x))\to-\infty$ and hence ${\Re F(x)\to0}$ as ${(0,+\infty)\ni x\to+\infty}$. Our next target is to prove the following claim.
		\medskip

		\noindent \textbf{Claim:} There exists a number $c \in \R$ such that $F(x)\to ic$ as $(0,+\infty)\ni x\to+\infty$.
		\medskip

		Note that from~\eqref{EQ-1} it follows that there exists a real number ${x_0>0}$ such that
		$$
		-|F'(x)|\le \Re F'(x) \le -\tfrac{1}2|F'(x)|\quad\text{for all~$~x\in[x_0,+\infty)$.}
		$$
		Taking this into account and recalling that ${\Re F(x)\to0}$ as ${x\to+\infty}$, we see that the integral $\int_{x_0}^{+\infty}\Re F'(x)\di x$ and hence also $\int_{x_0}^{+\infty}|F'(x)|\di x$ do converge. Therefore, $F$ is bounded on~$[x_0,+\infty)$. As a consequence, the limit set of~$F(x)$ as ${[x_0,+\infty)\ni x\to+\infty}$ is either a single point or a non-degenerate compact interval $I$ on the imaginary axis. To prove our claim, it is sufficient to rule out the latter alternative. For this purpose, fix some interior point~$ic_1$ of~$I$. Then there exists a sequence $(t_n)\subseteq{[x_0,+\infty)}$ tending to~$+\infty$  and such that
		\begin{equation}\label{EQ_diff-is-real}
			F(t_n)-ic_1\,\in\,\Real \quad\text{for all~$~n\in\Natural$.}
		\end{equation}
		Consider the self-map $F_*:=1/(F-ic_1)\in\Hol(\UH,\UH)$, which~--- as it is not difficult to see~--- satisfies
		$$
		\dhyp F_*(\zeta)=\dhyp F(\zeta)\,\to\,1\quad\text{as~$~\zeta\to\infty~$ non-tangentially.}
		$$
		Therefore, by the argument of \textsc{Steps~1 and~2} applied to~$F_*$ instead of~$F$, we have
		$$
		\frac{F'_*(x)}{|F'_*(x)|}\to\epsilon\in\{+1,-1\}\quad\text{as~$~(0,+\infty)\ni x\to+\infty$.}
		$$

		On the one hand, by~\eqref{EQ-1} and~\eqref{EQ_diff-is-real},
		$$
		\frac{F'_*(t_n)}{|F'_*(t_n)|}=-\frac{F'(t_n)}{|F'(t_n)|}\,\frac{|F(t_n)-ic_1|^2}{(F(t_n)-ic_1)^2}
		=-\frac{F'(t_n)}{|F'(t_n)|}~\to~1 \quad\text{as~$~n\to+\infty$},
		$$
		and hence ${\epsilon=1}$.

		On the other hand, if $ic_2\in I\setminus\{ic_1\}$ and if $(s_n)$ is a sequence in ${[x_0,+\infty)}$ tending to~$+\infty$ and such that $F(s_n)$ converges to~$ic_2$, then
		$$
		\frac{F'_*(s_n)}{|F'_*(s_n)|}=-\frac{F'(s_n)}{|F'(s_n)|}\,\frac{|F(s_n)-ic_1|^2}{(F(s_n)-ic_1)^2}
		~\to~\frac{|ic_2-ic_1|^2}{(ic_2-ic_1)^2}~=~-1 \quad\text{as~$~n\to+\infty$},
		$$
		and hence ${\epsilon=-1}$. This contradiction completes the proof of our claim.
		\medskip

		Consider $\widehat F:=1/(F-ic)\in\Hol(\UH,\UH)$, where $ic$ is the limit of~$F(x)$ as ${x\to+\infty}$. According to Lindel\"of's angular limit theorem, $\infty$~is boundary fixed point of~$\widehat F$. Since $\dhyp \widehat F=\dhyp F$,  one can repeat all the above arguments with $F$ replaced by~$\widehat F$. Moreover, note that the second alternative in \textsc{Step~2} cannot hold for~$\widehat F$ because by the claim proved above, this would imply that $\widehat F$ has a finite angular limit at $\infty$ instead of a boundary fixed point at~$\infty$.

		Thus,
		$$
		\anglim_{\zeta\to\infty}\widehat F(\zeta)~=~\infty\quad\text{and}
		\quad \anglim_{\zeta\to\infty}\frac{\widehat F'(\zeta)}{|\widehat F'(\zeta)|}~=~1.
		$$
		By Lemma~\ref{LM_isogonality}, it follows that
		$$ \angle \lim \limits_{\zeta \to \infty} \arg \frac{\widehat F(\zeta)}{\zeta}=0 \, .$$

		Passing back to the unit disk, we conclude
		$$
		\varphi(1):=\anglim_{z\to1}\varphi(z)~=~\frac{ic-1}{ic+1}\in\UC  \quad\text{and}\quad \anglim_{z\to1}\frac{1-\overline{\varphi(1)}\varphi(z)}{1-z}~=~0.
		$$
		The proof of the implication \ref{IT_FEQ_a}~$\Longrightarrow$~\ref{IT_TH-main1:semireg-cntct-pnt} of Theorem~\ref{TH-main1} is complete.
	\end{proof}

\section{Proof of Theorem~\ref{TH-main2} -- Part I: geometric point of view}\label{sec_proof-main2}

	In this section we treat the chain of implications \ref{IT_TH-main2:int-conv} $\Longrightarrow$ \ref{IT_TH-main2:regular-cntct-pnt} $\Longrightarrow$ \ref{IT_TH-main2:BetsakosKaramanlis-angular} $\Longrightarrow$ \ref{IT_TH-main2:BetsakosKaramanlis-radial} $\Longrightarrow$ \ref{IT_TH-main2:int-conv} in Theorem~\ref{TH-main2} and Addendum~\ref{ADD:TH-main2Condition(d)}.
	We first prove implication~\ref{IT_TH-main2:int-conv} $\Longrightarrow$ \ref{IT_TH-main2:regular-cntct-pnt} under an additional assumption. As in the proof of Theorem~\ref{TH-main1}, we work in the right half-plane model (see Section~\ref{SUB:half-plane}).

	\begin{proposition}\label{PR_weak}
		Suppose $F\in\Hol(\UH,\UH)$ has a boundary fixed point at~$\infty$ and $F$ is conformal at~$\infty$ in the weak sense. If
		\begin{equation}\label{EQ_int-conv-cond-weak}
			\int\limits_1^{+\infty}\big(1-\dhyp F(x) \big)\frac{\di x}{x}~<~+\,\infty,
		\end{equation}
		then ${F'(\infty)>0}$ (and hence, $F$ is conformal at~$\infty$ also in the strong sense).
	\end{proposition}

	Before we prove Proposition~\ref{PR_weak}, we observe that its converse is in fact much more elementary.
	\begin{lemma}\label{LM_elementary-implication}
		If $F\in\Hol(\UH,\UH)$ and $F'(\infty)>0$, then $F$ satisfies condition~\eqref{EQ_int-conv-cond-weak}.
	\end{lemma}

	\begin{proof}
		By the hypothesis,  $\lim_{x\to+\infty}F(x)/x\,=\,F'(\infty)\in(0,+\infty)$. Therefore,
		\begin{eqnarray*}
			\int\limits_1^{+\infty}\left(1-\frac{x\Re F'(x)}{\Re F(x)}\right)\, \frac{\di x}{x}
			&=&\big(\log x\,-\,\log\Re F(x)\big)\Big|_{x=1}^{x=+\infty}\\
			&=&\log \Re F(1)-\lim_{x\to+\infty}\log\frac{\Re F(x)}{x}\\
			&=&\log\Re F(1)\,-\,\log F'(\infty)~<~+\infty.
		\end{eqnarray*}
		Since $0\,\le\, 1-\dhyp F(x)\,\le\,1-x\Re F'(x)/\Re F(x)$ for all ${x\ge0}$, \eqref{EQ_int-conv-cond-weak}~follows immediately.
	\end{proof}

	\begin{proof}[\proofof{Proposition~\ref{PR_weak}.}]
		Let $\gamma:[0,+\infty)\to\UH$ be the smooth injective curve from Proposition~\ref{PR_Real_axis_in_image}. We first show that~\eqref{EQ_int-conv-cond-weak} is equivalent to
		\begin{equation}\label{EQ_int-conv-cond-weak-for-gamma}
			\int\limits_{1}^{+\infty}\big(1-\dhyp F(\gamma(s))\big) \frac{\di s}{s}<+\infty .
		\end{equation}
		Indeed, in view of~\eqref{eq:Golusin2},
		$$ e^{-2\hypdist_\UH(\gamma(s),s )} \leq \frac{1-\dhyp F(\gamma(s))}{1-\dhyp F(s)} \leq e^{2\hypdist_\UH(\gamma(s),s )},\quad s\geq1.$$
		Further, by Proposition~\ref{PR_Real_axis_in_image}\ref{IT_gamma-bdd-hyperbolic-distance}, there is $M>0$ such that $\hypdist_\UH(\gamma(s),s)\leq M$ for all $s\geq1$. Hence, $$e^{-2M} \int\limits_{1}^{+\infty} \big(1-\dhyp F(s)\big) \frac{\di s}{s} \leq \int\limits_1^{+\infty} \big(1-\dhyp F(\gamma(s))\big) \frac{\di s}{s} \leq e^{2M} \int\limits_{1}^{+\infty} \big(1-\dhyp F(s)\big) \frac{\di s}{s}, $$
		which implies the equivalence of~\eqref{EQ_int-conv-cond-weak} and~\eqref{EQ_int-conv-cond-weak-for-gamma}.

		Next, by Proposition~\ref{PR_Real_axis_in_image}\ref{IT_gamma_Re}, there exists ${c>0}$ such that $cs\leq\Re\gamma(s)\leq s$ for all $s\geq1$. Further, by Proposition~\ref{PR_Real_axis_in_image}\ref{IT_gamma_arc-len-param} we have $|\gamma'(s)|=1$ for all $s\geq0$. Together with~\eqref{EQ_int-conv-cond-weak-for-gamma} this implies
		\begin{eqnarray*}
			+\infty~>~K&:=&\frac{1}{c}   \int\limits_{1}^{+\infty} \big(1-\dhyp F(\gamma(s))\big) \frac{\di s}{s} ~\geq~   \int\limits_{1}^t \big(1-\dhyp F(\gamma(s))\big) \frac{\di s}{ \Re \gamma(s)} \nonumber \\
			&=&  \int\limits_{1}^t \frac{|\gamma'(s)|}{ \Re \gamma(s)} \di s - \int\limits_{1}^t  \frac{|F^{\prime}(\gamma(s))| }{\Re F(\gamma(s))} |\gamma^{\prime}(s)| \di s
			~\geq~   \int\limits_{1}^t \frac{\Re \gamma'(s)}{ \Re \gamma(s)} \di s - \int\limits_{1}^t  \frac{F^{\prime}(\gamma(s)) \gamma^{\prime}(s)}{ F(\gamma(s))} \di s \nonumber \\
			&=& \log \Re \gamma(t) -\log F(\gamma(t)) +\log F(\gamma(1))- \log \Re \gamma(1)
		\end{eqnarray*}
		for all $t\geq 1$. In the last inequality, we have used the fact that ${s\mapsto F(\gamma(s))}$ is real-valued and strictly increasing.
		On the one hand, the above inequality ~implies that there exists some constant $C>0$ such that
		\begin{equation}\label{EQ_fromZero}
			\frac{F(\gamma(t))}{\Re \gamma(t)} \ge C\quad\text{for all~$~t\ge t_0$}.
		\end{equation}
		On the other hand, by the very definition, $F'(\infty)=\anglim_{\zeta\to\infty}F(\zeta)/\zeta$, where the limit is known to exist and belongs to~${[0,+\infty)}$, see e.g.~\cite[\S26]{Valiron:book}.
		In view of properties~\ref{IT_gamma_Re} and~\ref{IT_gamma_Im} in Proposition~\ref{PR_Real_axis_in_image}, it follows that
		$$
		\lim_{t\to+\infty}\frac{F(\gamma(t))}{\Re \gamma(t)}=\lim_{t\to+\infty}\frac{F(\gamma(t))}{\gamma(t)}\cdot
		\lim_{t\to+\infty}\frac{\gamma(t)}{\Re\gamma(t)}=F'(\infty)\cdot 1,
		$$
		and hence, recalling~\eqref{EQ_fromZero}, we see that~$F'(\infty)>0$ as desired.
	\end{proof}

	We can now turn to the proof that the conditions \ref{IT_TH-main2:int-conv}, \ref{IT_TH-main2:BetsakosKaramanlis-angular} and \ref{IT_TH-main2:regular-cntct-pnt} listed in Theorem~\ref{TH-main2} and condition~\ref{IT_TH-main2:BetsakosKaramanlis-radial} in Addendum~\ref{ADD:TH-main2Condition(d)} are pairwise equivalent.
	We (partially) use the half-plane model (see Section~\ref{SUB:half-plane}) and work with $F\in\Hol(\UH,\UH)$.

	\begin{proof}[\proofof{Theorem~\ref{TH-main2}: \ref{IT_TH-main2:int-conv}$~\Longrightarrow~$\ref{IT_TH-main2:regular-cntct-pnt}}]
		Using Proposition~\ref{prop:4}, we see that condition~\ref{IT_TH-main2:int-conv} implies that $\angle\lim_{z\to\sigma}\dhyp \varphi(z)=1$. By Theorem~\ref{TH-main1}, this is equivalent to $\varphi$ being conformal at~$\sigma$ in the weak sense. Then, passing to the right half-plane, we can apply Proposition~\ref{PR_weak}, which proves~\ref{IT_TH-main2:regular-cntct-pnt}.
	\end{proof}

	\begin{proof}[\proofof{Theorem~\ref{TH-main2}: \ref{IT_TH-main2:regular-cntct-pnt}~$\Longrightarrow$~\ref{IT_TH-main2:BetsakosKaramanlis-angular}}] Without loss of generality, take ${\sigma=1}$. Then, passing to the right half-plane setting, as explained in Section~\ref{SUB:half-plane}, yields a self-map $F\in\Hol(\UH,\UH)$ satisfying ${F'(\infty)>0}$. Replacing $F$ by $F/F'(\infty)$, we may suppose that ${F'(\infty)=1}$, i.e.
		\begin{equation}\label{EQ_tangent-to-identity}
			\anglim_{\zeta\to\infty}\frac{F(\zeta)}{\zeta}=1.
		\end{equation}
		For $\zeta_0\in\UH$, we denote by $T_{\zeta_0}$ the automorphism of~$\UH$ given by $T_{\zeta_0}(\zeta):={(\zeta-i\Im \zeta_0)/\Re \zeta_0}$.
		Using the Schwarz\,--\,Pick Lemma, the triangle inequality and the invariance of the hyperbolic distance w.r.t. automorphisms, we obtain
		\begin{eqnarray*}
			0~\le~\hypdist_\UH\big(\zeta_1,\zeta_2\big)-\hypdist_\UH\big(F(\zeta_1),F(\zeta_2)\big) &\le &
			\hypdist_\UH\big(F(\zeta_1),\zeta_1\big) + \hypdist_\UH\big(F(\zeta_2),\zeta_2\big)
			\\
			& = &
			\hypdist_\UH\big(T_{\zeta_1}(F(\zeta_1)),1\big) + \hypdist_\UH\big(T_{\zeta_2}(F(\zeta_2)),1\big)
		\end{eqnarray*}
		for all $\zeta_1,\zeta_2\in\UH$. Now \ref{IT_TH-main2:BetsakosKaramanlis-angular} follows from the observation that thanks to~\eqref{EQ_tangent-to-identity}, we have
		$$
		\anglim_{\zeta\to\infty}\hypdist_\UH\big(T_{\zeta}(F(\zeta)),1\big)=0.
		$$
		Indeed,
		$$
		\big|T_{\zeta}(F(\zeta))-1\big|~=~\frac{|\zeta|}{\Re \zeta}\left|\frac{F(\zeta)}{\zeta}-1\right|~\to~0\qquad\text{as $~\zeta\to\infty~$ non-tangentially,}
		$$
		and it remains to recall that on each compact subset of $\UH$, the euclidean distance is equivalent to~$\hypdist_\UH$; see~\cite[Remark 1.3.6]{BCD-Book}.
	\end{proof}

	\begin{proof}[\proofof{Theorem~\ref{TH-main2}: \ref{IT_TH-main2:BetsakosKaramanlis-angular}~$\Longrightarrow$~\ref{IT_TH-main2:BetsakosKaramanlis-radial}}\nopunct] is trivial.
	\end{proof}

	\begin{proof}[\proofof{Theorem~\ref{TH-main2}: \ref{IT_TH-main2:BetsakosKaramanlis-radial}~$\Longrightarrow$~\ref{IT_TH-main2:int-conv}}] In the half-plane model, condition~\ref{IT_TH-main2:int-conv} transforms to $${I:=\lim_{y\to+\infty} I(y)}<+\infty \, , $$   where
		$$
		I(y):=\int_1^{y}\big(1-\dhyp F(x)\big){\lambda_\UH(x)}\di x,\qquad y\ge1.
		$$
		The integrand is non-negative by the Schwarz\,--\,Pick Lemma. Moreover, since every segment ${[y_1,y_2]\subseteq(0,+\infty)}$ is a hyperbolic segment in~$\UH$, we have
		\begin{equation*}\label{EQ_I(a)I(b)}
			0~\le~I(b)-I(a)~=~\hypdist_\UH(a,b)-\ell_\UH\big(F([a,b])\big)~\le~\hypdist_\UH\big(a,b\big)-\hypdist_\UH\big(F(a),F(b)\big)
		\end{equation*}
		for all $b \ge a \ge 1$. Thus, if condition~\ref{IT_TH-main2:BetsakosKaramanlis-radial} is satisfied, then $I(y)$ has a finite limit as ${y\to+\infty}$, i.e. \ref{IT_TH-main2:int-conv} holds.
	\end{proof}

	\section{Proof of Theorem~\ref{TH-main2} -- Part~II: operator theory point of view}\label{SEC:Operatortheory}

	We now turn to the operator-theoretic aspect of Theorem~\ref{TH-main2}. Recall that Theorem~\ref{TH-main2} gives a characterization of a self-map $\varphi\in\Hol(\UD,\UD)$ conformal at a boundary point ${\sigma\in\partial\UD}$ in the strong sense, i.e. having angular limit~${\anglim_{z\to\sigma}\varphi(z)}\in\UC$ and finite angular derivative~$\varphi'(\sigma)$. The same situation occurs in the classical Julia\,--\,Wolff\,--\,Carath\'{e}odory theorem, which admits a purely operator-theoretic approach. This was first observed in~\cite{Sarason1988}. We will sketch this approach in the following subsection, and use it afterwards in order to place Theorem~\ref{TH-main2} into an operator-theoretic context.

	\subsection{De Branges\,--\,Rovnyak spaces}

	Every self-map $\varphi\in\Hol(\UD,\UD)$ induces a certain Hilbert space~${\Hbspace\subseteq\Hol(\UD,\C)}$,
	discovered and studied for the first time by de~Branges and Rovnyak~\cite{deBR} and commonly referred to by their names. For any $\varphi\in\Hol(\UD,\UD)$, the function
	\begin{equation}\label{EQ_kernel-explicit-form}
		k^{\varphi}(z,w)=k^\varphi_w(z):=\frac{1-\overline{\varphi(w)}\varphi(z)}{1-\overline{w}z}\quad z,w\in\UD,
	\end{equation}
	is positive definite. The de Branges\,--\,Rovnyak space~$\Hbspace$ can be defined as the unique RKHS (reproducing kernel Hilbert space) for which $k^\varphi$ is the reproducing kernel. 
        The latter means that
	\begin{equation}\label{EQ_reproducing-property}
		k^\varphi_w\in\Hbspace~\text{~and~}~f(w)=\langle f, k_w^\varphi \rangle_\varphi \quad \text{ for all }~f\in \Hbspace~\text{ and all }~w\in\UD,
	\end{equation}
	where $\langle \cdot, \cdot \rangle_\varphi$ stands for the inner product in~$\Hbspace$.
        For further details and a solid introduction into de Branges\,--\,Rovnyak space and RKHSs, we refer the interested reader to the monographs \cite{FM1,FM2,RKHSp,Sarason1994}. 
        
	In what follows, we will not need any explicit formula for~$\langle \cdot, \cdot \rangle_\varphi$.
	However, it is worth calculating the norm of the reproducing kernel:
	\begin{equation}\label{EQ_kernel-norm}
		\left\Vert k_w^\varphi \right\Vert_\varphi^2 = \langle k_w^\varphi, k_w^\varphi \rangle_\varphi =k_w^\varphi(w) = \frac{1-|\varphi(w)|^2}{1-|w|^2}
	\end{equation}
	for any $w\in\UD$. Note that the last expression in the above equality is Julia's quotient, which appears in the  Julia\,--\,Wolff\,--\,Carath\'{e}odory theorem.

	\subsection{Proof of Theorem~\ref{TH-main2}: operator-theoretic aspects}

Clearly, \ref{IT_TH-main2:operator-theory-liminf1} $\Longrightarrow$ \ref{IT_TH-main2:operator-theory}.
	In order to complete the proof of Theorem~\ref{TH-main2} and Addendum~\ref{ADD:TH-main2Condition(d)}, we are left to show 
        \ref{IT_TH-main2:BetsakosKaramanlis-angular}$\,\Longleftrightarrow\,$\ref{IT_TH-main2:operator-theory-liminf1} and \ref{IT_TH-main2:operator-theory}$\,\Longrightarrow\,$\ref{IT_TH-main2:regular-cntct-pnt}.

	\begin{proof}[\proofof{Theorem~\ref{TH-main2}: \ref{IT_TH-main2:BetsakosKaramanlis-angular}~$\Longleftrightarrow$~\ref{IT_TH-main2:operator-theory-liminf1}.}]
		Let $\varphi\in\Hol(\UD,\UD)$. The key observation for this part of the proof is that, in terms of the pseudo-hyperbolic distance, condition~\ref{IT_TH-main2:BetsakosKaramanlis-angular} is equivalent to the ``invariant version'' of Julia's quotient tending to~1, i.e.
		\begin{equation}\label{IT_TH-main2:pseudo-distance-quotient}
			\angle\lim_{z,w\to\sigma}\frac{1-\varrho_{\UD}\big(\varphi(z),\varphi(w)\big)^2}{1-\varrho_{\UD}(z,w)^2} = 1.
		\end{equation}
		Indeed, for any $z,w\in\UD$,
		$$
		0 ~\le~
		\log\frac{1-\varrho_{\UD}\big(\varphi(z),\varphi(w)\big)^2}{1-\varrho_{\UD}(z,w)^2} ~\le~
		\hypdist_\UD(z,w)-\hypdist_\UD(\varphi(z),\varphi(w))
		~\le~ \sqrt{\frac{1-\varrho_{\UD}\big(\varphi(z),\varphi(w)\big)^2}{1-\varrho_{\UD}(z,w)^2}-1},
		$$
		which is not difficult to check taking into account that by the Schwarz\,--\,Pick Lemma, $${0\le \varrho_\UD(\varphi(z),\varphi(w))} \le{\varrho_\UD(z,w) <1}.$$
		Now using the explicit expression of $\varrho_\UD$ and the properties of $k_z^\varphi$, i.e. \eqref{EQ_pseudohyperbolic-distance},  \eqref{EQ_kernel-explicit-form}, \eqref{EQ_reproducing-property}, and~\eqref{EQ_kernel-norm}, it is easy to see that
		\begin{equation}\label{EQ_kernel_condition_is_invariant_Julia}
			\angle\lim_{z,w\to\sigma}\frac{\left|\langle k_z^\varphi, k_w^\varphi \rangle_\varphi \right\vert^2}{\left\Vert k_z^\varphi \right\Vert_\varphi^2 \left\Vert k_w^\varphi \right\Vert_\varphi^2}
			~=~			\angle\lim_{z,w\to\sigma}\frac{1-\varrho_{\UD}(z,w)^2}{1-\varrho_{\UD}\big(\varphi(z),\varphi(w)\big)^2}.
		\end{equation}
		Thus, condition~\eqref{IT_TH-main2:pseudo-distance-quotient} --- and hence also~\ref{IT_TH-main2:BetsakosKaramanlis-angular} --- is equivalent to~\ref{IT_TH-main2:operator-theory-liminf1}.
	\end{proof}

	\begin{proof}[\proofof{Theorem~\ref{TH-main2}: \ref{IT_TH-main2:operator-theory}$~\Longrightarrow~$\ref{IT_TH-main2:regular-cntct-pnt}.}]
		{This part of the proof closely follows the operator-theoretic approach of Sarason to the Julia-Wolff-Carath\'{e}odory theorem, see e.g.~\cite[Theorem 21.1]{FM2}. In particular, we use that Julia's quotient is related to the norm of the kernel functions in~$\mathcal{H}(\varphi)$ by~\eqref{EQ_kernel-norm}.

		In order to take advantage of weak precompactness of bounded sets, it is convenient to work with the normalized reproducing kernels $\hat{k}_z^\varphi \in \Hbspace$, $z\in\D$, that is with the functions
		\[\hat{k}_z^\varphi:=\frac{k_z^\varphi}{\Vert k_z^\varphi \Vert_\varphi}.\]
		Clearly, $\Vert \hat{k}_z^\varphi \Vert_\varphi=1$. Choose a sequence $(z_n)\subseteq\UD$ converging to~$\sigma$ non-tangentially. By condition~\ref{IT_TH-main2:operator-theory} there is $w_0\in\UD$ such that
		\begin{equation}\label{EQ_IT_TH-main2:operator-theory-normalized}
			\angle \liminf_{n\to \infty} \left\vert\langle \hat{k}_{z_n}^\varphi, \hat{k}_{w_0}^\varphi \rangle_\varphi \right\vert > 0.
		\end{equation}
		Since $(\hat{k}_{z_n}^\varphi)$ is a bounded sequence in $\Hbspace$, we can assume, by passing to a suitable subsequence, if necessary, that there is ${q\in\Hbspace}$ such that ${\hat{k}_{z_n}^\varphi \weaklyto q}$ as ${n\to+\infty}$, i.e. for any~${f\in \Hbspace}$,
		$$
		\langle f, \hat{k}_{z_n}^\varphi \rangle_\varphi \to \langle f, q \rangle_\varphi \quad \text{as~}~ n\to+\infty.
		$$
		In view of~\eqref{EQ_IT_TH-main2:operator-theory-normalized} we therefore obtain
		\[ |q(w_0)|=\left\vert\langle q,\hat{k}_{w_0}^\varphi \rangle_\varphi\right\vert  >0.\]
		Since $(\varphi(z_n))\subseteq\overline{\UD}$ is a bounded sequence, we can assume --- by passing to a (further) subsequence, if necessary --- that there is ${\lambda\in\overline{\UD}}$ such that ${\varphi(z_n)\to\lambda}$ as ${n\to+\infty}$. This yields
		\[\lim_{n\to\infty} k_{z_n}^\varphi(w_0)
		=
		\lim_{n\to\infty}\frac{1-\overline{\varphi(z_n)}\varphi(w_0)}{1-\overline{z_n}w_0}
		=
		\frac{1-\overline{\lambda}w_0}{1-\overline{\sigma}w_0} \neq 0.\]
		Hence,
		\begin{align*}
			\lim_{n\to+\infty}\left(\frac{1-|z_n|^2}{1-|\varphi(z_n)|^2}\right)^{1/2}
			&=~
			\lim_{n\to+\infty}\left\Vert k_{z_n}^\varphi \right\Vert^{-1}
			~=~
			\lim_{n\to+\infty}\left\vert\frac{\hat{k}_{z_n}^\varphi(w_0)}{k_{z_n}^\varphi(w_0)}\right\vert
			\\ ~&=~
			\lim_{n\to+\infty}\left\vert\frac{\langle \hat{k}_{z_n}^\varphi, k_{w_0}^\varphi \rangle}{k_{z_n}^\varphi(w_0)}\right\vert
			~=~
			\left|\langle q, k_{w_0}^\varphi \rangle \frac{1-\overline{\sigma}w_0}{1-\overline{\lambda}w_0}\right|
			~=~
			\left| q(w_0) \, \frac{1-\overline{\sigma}w_0}{1-\overline{\lambda}w_0}\right| ~>~ 0.
		\end{align*}
		It follows that
		\[\lim_{n\to+\infty} \frac{1-|\varphi(z_n)|}{1-|z_n|} < +\infty. \]
		By the classical Julia\,--\,Wolff\,--\,Carath\'{e}odory Theorem, this implies that ${\anglim_{z\to\sigma}\varphi(z)}$ exists and belongs to~$\UC$ and that the angular derivative $\varphi'(\sigma)$ exists finitely; in other words, $\varphi$ is conformal at~$\sigma$ in the strong sense, as desired. The proof of Theorem~\ref{TH-main2} and  Addendum~\ref{ADD:TH-main2Condition(d)} is now complete.}
	\end{proof}

	\begin{remark}
		\begin{enumerate}
			\item We wish to mention that  Sarason's result (e.g.~\cite[Theorem 21.2]{FM2}) directly shows that condition~\ref{IT_TH-main2:regular-cntct-pnt} implies condition \ref{IT_TH-main2:operator-theory-liminf1}.

			\item The value~1 of the limit in condition~\ref{IT_TH-main2:operator-theory-liminf1}, i.e.\
			\[\angle \lim_{z,w\to \sigma}\left\vert\langle \hat{k}_z^\varphi, \hat{k}_w^\varphi \rangle_\varphi \right\vert
			= 1, \]
			has a geometric interpretation: the inner product can be understood as the cosine of the angle between the normalized kernel functions $\hat{k}_z^\varphi$ and $\hat{k}_w^\varphi$. This way, the limit demands that $\hat{k}_z^\varphi$ and $\hat{k}_w^\varphi$ are getting parallel as $z$ and $w$ approach $\sigma$. Note that condition~\ref{IT_TH-main2:operator-theory-liminf1} can be equivalently stated as
			\[\angle \lim_{z,w\to \sigma} \sqrt{1-\left\vert\langle \hat{k}_z^\varphi, \hat{k}_w^\varphi \rangle \right\vert^2}=0.\]
			This last expression for general reproducing kernel functions appears in the literature in combination with a certain distance function (see e.g.~\cite{Agler2002,ARSW,Sawyer,Seip}: if $\mathcal{H}$ is an arbitrary reproducing kernel Hilbert space on a set~$X$ and $\hat{k}_x, \hat{k}_y$ are the corresponding normalized kernel functions for $x,y\in X$, then
			\[\delta_\mathcal{H}:X\times X,\quad \delta_\mathcal{H}(x,y)=\sqrt{1-\left\vert\langle \hat{k}_x, \hat{k}_y \rangle \right\vert^2}\]
			defines a pseudo-metric. In particular, if $\mathcal{H}$ is the Hardy space $H^2(\D)$, then $\delta_{H^2(\D)}=\rho_\UD$. For treating interpolation problems one is interested in $\delta_\mathcal{H}$ being bounded away from zero. Our condition~\ref{IT_TH-main2:operator-theory-liminf1} somehow deals with the reverse situation.

\item \label{IT_Rem_limsup_condition} In order to establish condition~\ref{IT_TH-main2:regular-cntct-pnt} from condition~\ref{IT_TH-main2:operator-theory} it suffices to know that
			\begin{equation*}
				\limsup_{z\to\sigma}\left\vert\langle \hat{k}_z^\varphi,\hat{k}_{w_0}^\varphi \rangle \right\vert >0
			\end{equation*}
			for some $w_0\in\UD$. This is notably less information than the behaviour of the inner product along two sequences tending to the boundary simultaneously, which we assume in condition~\ref{IT_TH-main2:operator-theory}. In particular, the exact value of this double limit from condition~\ref{IT_TH-main2:operator-theory-liminf1} does not play~a role as long as it is strictly positive. However, Theorem~\ref{TH-main2} has a self-improving effect: once the limit inferior is strictly positive, i.e.\ condition~\ref{IT_TH-main2:operator-theory} holds, it already follows that the double angular limit is equal to~1 as  condition~\ref{IT_TH-main2:operator-theory-liminf1} states.
		\end{enumerate}
              \end{remark}

Taking the last remark into account and what we have observed in the beginning of the proof of   \ref{IT_TH-main2:BetsakosKaramanlis-angular}$\,\Longleftrightarrow\,$\ref{IT_TH-main2:operator-theory-liminf1} we can add the following conditions to the list of equivalent conditions of Theorem~\ref{TH-main2}.
	\begin{addendum}[\textsl{Further equivalent conditions to conformality at the boundary in the strong sense}]
		\label{ADD_3}
		A self-map $\varphi\in\Hol(\UD,\UD)$ is conformal at the point~$\sigma\in\partial\D$ in the strong sense if and only if  one~--- and hence each~--- of the following conditions holds:
		\setcounter{alphlistcounter}{1}
		\begin{alphlistPP}
			\item $\displaystyle\angle\lim_{z,w\to\sigma}\frac{1-\varrho_{\UD}\big(\varphi(z),\varphi(w)\big)^2}{1-\varrho_{\UD}(z,w)^2} =1 $;
			\medskip
			\item[\textbf{(b''')}] $\displaystyle\liminf_{z\to\sigma}\frac{1-\varrho_{\UD}\big(\varphi(z),\varphi(w_0)\big)^2}{1-\varrho_{\UD}(z,w_0)^2} <+\infty \quad \text{for some } w_0\in\D$;
			\medskip
			\item $\limsup_{z\to\sigma} \left\vert\langle \hat{k}_{z}^\varphi, \hat{k}_{w_0}^\varphi \rangle_\varphi\right\vert > 0\quad \text{for some } w_0\in\D$.
         \end{alphlistPP}
	\end{addendum}

 \begin{remark} \label{REM:TH-main2(c)VS(d)}
Of course, the fact that  \textbf{(b''')} implies strong conformality at $\sigma$ can also be shown directly with the help of Julia's lemma.
Condition  \textbf{(b''')} in terms of $\hypdist_\D$ instead of $\varrho_\D$ takes the form
\begin{equation} \label{EQ:LiminfAngDer}
  \liminf_{z \to \sigma} \left[ \hypdist_\D \big(z,w_0\big)-\hypdist_\D\big(\varphi(z),\varphi(w_0) \big)\right]<+\infty \quad \text{for some } w_0\in \D \, ,
  \end{equation}
which is clearly equivalent to
\begin{equation}\label{EQ_Julia-metric}
 \liminf_{z \to \sigma} \left[ \hypdist_\D \big(z,w_0\big)-\hypdist_\D\big(\varphi(z),w_0 \big)\right]<+\infty \quad \text{for some } w_0\in \D \, .
\end{equation}
The fact that the latter condition is equivalent to the conformality of $\varphi$ in the strong sense at~$\sigma$ is Theorem~1.7.8 in~\cite{BCD-Book}. As in the operator-theoretic picture the self-improving effect of Theorem~\ref{TH-main2} says that
\eqref{EQ:LiminfAngDer} is equivalent to
  $$\displaystyle \angle \lim \limits_{z,w \to \sigma}\big(\,\hypdist_\UD\big(z,w\big)-\hypdist_\UD\big(\varphi(z),\varphi(w)\big)\,\big)\,=\,0\, .$$
\end{remark}

	\section{Applications to Holomorphic Dynamical Systems}\label{SEC:ProofTH3}
	\subsection{Iteration theory and backward dynamics}\label{SUB_iteration}
			 In this subsection we recall basic definitions and facts from the theory of holomorphic iteration in the unit disk. For a complete up-to-date overview of the topic and proofs of the statements given below, we refer interested readers to~\cite{Marco} and in particular, to Chapter~4 of that monograph. A fundamental result concerning forward dynamics of a holomorphic self-map ${\varphi\in\Hol(\UD,\UD)}$, known as the Denjoy\,--\,Wolff Theorem, states that except for the trivial case of non-euclidean rotations $\varphi={\psi^{-1}\circ(z\mapsto e^{i\theta}z)\circ\psi}$, ${\psi\in\Aut(\UD)}$, ${\theta\in\Real}$, the sequence $(\varphi^{\circ n})$ formed by the iterates ${\varphi^{\circ 0}:=\id_{\UD}}$, ${\varphi^{\circ n}:=\varphi\circ\varphi^{\circ(n-1)}}$ for~${n\in\Natural}$, converges locally uniformly in~$\UD$ to a constant function ${g\equiv\tau_\varphi\in\overline{\UD}}$. The point $\tau_\varphi$ is called the \emph{Denjoy\,--\,Wolff} point of~$\varphi$.

			If $\tau_\varphi\in\UD$, then it is immediate to see that $\tau_\varphi$ is a fixed point of~$\varphi$. The same can be stated in the case where ${\tau_\varphi\in\UC}$, if we extend the notion of a fixed point to the boundary as follows: ${\sigma\in\UC}$ is said to be a \emph{boundary fixed point} for~$\varphi$ if the angular limit ${\anglim_{z\to\sigma}\varphi(z)}:=\varphi(\sigma)$ exists and coincides with~$\sigma$. In contrast to interior fixed points, which a holomorphic self-map ${\varphi\neq\id_\UD}$ has at most one, there can be many~--- and even infinitely many~--- boundary fixed points. Moreover, it is known that the angular derivative $\varphi'(\sigma):={\anglim_{z\to\sigma}\big(\varphi(z)-\sigma\big)/(z-\sigma)}$ exists at any boundary fixed point~$\sigma$; however, it can be infinite. More precisely, $\varphi'(\sigma)\in{(0,+\infty)\cup\{\infty\}}$.

			If ${\varphi'(\sigma)=\infty}$, then the boundary fixed point~$\sigma$ is called \emph{super-repulsive}. Otherwise, i.e. in case $\varphi'(\sigma)\in{(0,+\infty)}$, the boundary fixed point~$\sigma$ is said to be \emph{regular}. If ${\tau_\varphi\in\UC}$, then ${\varphi'(\tau_\varphi)\in(0,1]}$. At the same time, for any boundary regular fixed point~$\sigma$ other than the Denjoy\,--\,Wolff point~$\tau_\varphi$, we have ${\varphi'(\sigma)>1}$. Such points are often referred to as \emph{repulsive fixed points}.

			Observe that in our terminology introduced in Section~\ref{SEC:MainResults}, a boundary fixed point is regular if and only if the self-map is conformal at that point in the strong sense. Therefore, it seems rather natural to introduce  the following definition as well.
			\begin{definition}
				A boundary fixed point~$\sigma$ of a holomorphic self-map $\varphi\in\Hol(\UD,\UD)$ is said to be \emph{semiregular} if $\varphi$ is conformal at~$\sigma$ in the weak sense.
			\end{definition}

            \begin{remark}
            It should be noted that this is not a standard definition in the literature.
                \end{remark}

			Understanding the asymptotic behaviour of the forward orbits ${w_n:=\varphi^{\circ n}(w_0)}$ as ${n\to+\infty}$ has been classically regarded as the primary goal of holomorphic dynamics. However, in recent years the study of the \textit{backward} dynamics has become another focal point of research.

			\begin{definition}
				Let $\varphi\in\Hol(\UD,\UD)$ and $z_0\in\UD$. An infinite sequence ${(z_n)\subset\UD}$ is said to be a \emph{backward orbit} for~$\varphi$ with the \emph{initial point}~$z_0$, if ${z_{n-1}=\varphi(z_n)}$ for all~${n\in\Natural}$. A backward orbit~$(z_n)$ is said to be \emph{regular}, if $(z_n)$ has no accumulation point\footnote{It is worth pointing out that a backward orbit~$(z_n)$ can have an accumulation point in~$\UD$ only in two simple situations: when $\varphi$ is a non-euclidean rotation of~$\UD$, or when ${z_n=\tau_\varphi\in\UD}$ for all ${n\in\Natural}$; see e.g. \cite[Proposition~4.9.3]{Marco}.}  in~$\UD$ and if it is of bounded hyperbolic step, i.e. ${\sup_{n\in\Natural}\hypdist_\UD(z_n,z_{n+1})<+\infty}$.
			\end{definition}

			\begin{remark}\label{RM_step}
				According to the Schwarz\,--\,Pick Lemma, for any backward orbit $(z_n)$,
				$$
				\hypdist_\UD(z_n,z_{n+1})\le \hypdist_\UD(z_m,z_{m+1})\le \rho:=\lim_{k\to+\infty}\hypdist_\UD(z_k,z_{k+1})
				$$
				for all $n,m\in\N_0$ with $n\le m$.
			\end{remark}

			Repulsive fixed points turn out to play an important role for backward dynamics similar to that of the Denjoy\,--\,Wolff point~$\tau_\varphi$ in the forward dynamics. In particular, every regular backward orbit converges to a repulsive fixed point, with the only possible exception: if ${\tau_\varphi\in\UC}$ and ${\varphi'(\tau_\varphi)=1}$, then there can exist regular backward orbits converging to~$\tau_\varphi$. For the exact statement of this result and further details on backward dynamics in~$\UD$, we refer interested readers to \cite[Section~4.9]{Marco} and to paper \cite{Bracci2003, Poggi-Corradini, Poggi-Corradini2003}. A sort of converse statement is contained in the following result by Poggi-Corradini.

			When it comes to questions related to boundary fixed points, it proves to be convenient to pass with the help of a suitable conformal map to the half-plane setting (see Section~\ref{SUB:half-plane}), with the fixed point under consideration being moved to~$\infty$. Clearly, all the definitions and results mentioned above can be reformulated for holomorphic self-maps ${f\in\Hol(\UH,\UH)}$.

		\begin{proposition}[\protect{\cite[Theorem 1.2]{Poggi-Corradini}}]\label{PR_Premodel_UH}
			Suppose that $f \in \Hol(\UH, \UH)$ has a repulsive fixed point at $\infty$. Then there exists $F\in \Hol(\UH, \UH)$ that satisfies the functional equation
			\begin{equation}\label{EQ_premodel-eq-in-the-proof}
				f\big(F(\zeta)\big)\,=\,F\big(f'(\infty)\zeta\big),\quad \zeta\in\UH.
			\end{equation}
			and has a boundary semiregular fixed point at~$\infty$. The map $F$ is unique up to precomposition with a M\"obius transformation of the form $\zeta\mapsto \alpha \zeta$, where $\alpha \in (0, +\infty)$.
			\end{proposition}
            In fact, aside from a finite number of initial terms, every regular backward orbit $(w_n)$ of~$f$ converging to~$\infty$ is of the form $w_n=F\big(f'(\infty)^{-n}\zeta_0\big)$ with a suitable ${\zeta_0\in\UH}$. In a sense, this is a reason why the triple $(\UH, F, \zeta \mapsto f^{\prime}(\infty) \zeta)$  is usually called a \emph{pre-model} for~$f$ at~$\infty$. If the boundary fixed point of~$F$ at $\infty$ is not only semiregular, but also regular, i.e. if ${F^{\prime}(\infty)>0}$, then the pre-model of~$f$ at~$\infty$ is said to be \emph{regular}.

            To conclude, we mention that the backward iteration has been studied in other contexts as well: for holomorphic flows in~$\UD$ with continuous time \cite{BackwSemigr,ShoiElinZalc,MK-KZ2022}, for discrete iteration in several complex variables \cite{Abate_Raissy_2011, Arosio2017, AB2016,  AG2019}, as well as in a much more general setting of non-expanding maps in Gromov hyperbolic metric spaces~\cite{AFGK2024}.

		\subsection{Proof of Theorem~\ref{TH-main3} and Corollary~\ref{CR_fromTH-main3}}

		In the current subsection, we again work in the right half-plane setting, see Section~\ref{SUB:half-plane}.
		Before proving Theorem \ref{TH-main3}, we provide a discrete analogue of condition~\eqref{EQ_int-conv-cond-weak}.

		\begin{proposition}\label{PR_cont<=>discrete}
			Let $F\in\Hol(\UH,\UH)$. Fix some ${\zeta_0\in\UH}$ and some ${\lambda>0}$. Then condition~\eqref{EQ_int-conv-cond-weak} holds if and only if
			\begin{equation}\label{EQ_main-cond-discrete_form}
				\sum_{n=1}^{+\infty} \big(1-\dhyp F(\lambda^n \zeta_0)\big)~<~+\infty.
			\end{equation}
		\end{proposition}
		\begin{proof}
			Using the half-plane analogue of Proposition \ref{prop:5} with with the automorphism of $\UH$ given by ${\zeta\mapsto\lambda \zeta}$, we obtain that \eqref{EQ_int-conv-cond-weak} is equivalent to the convergence of $\sum (1-\dhyp F(\lambda^n))$. Applying~\eqref{eq:Golusin2}, we get the desired equivalence, since the distance $\hypdist_{\UH}(\lambda^n,\lambda^n \zeta_0)$ does not depend on~$n$.
		\end{proof}

		\begin{proof}[\proofof{Theorem~\ref{TH-main3}}]
			Suppose $f \in \Hol(\UH, \UH)$ has a repulsive fixed point at~$\infty$. According to Proposition~\ref{PR_Premodel_UH}, there exists some ${F \in \Hol(\UH, \UH)}$ that has a semiregular fixed point at~$\infty$ and satisfies~\eqref{EQ_premodel-eq-in-the-proof}. Differentiating both parts of~\eqref{EQ_premodel-eq-in-the-proof}, we have
			\begin{equation}\label{EQ_premodel-der}
				f'\big(F(\zeta)\big)F'(\zeta)={ f^{\prime}(\infty) F'(f^{\prime}(\infty)\zeta)} \,
				\quad \text{for every } \zeta\in\UH.
			\end{equation}
				Let $(w_n)$ be a regular backward orbit of~$f$ converging to~$\infty$. Removing if necessary a finite number of terms, we have $w_n=F(\lambda^{n} \zeta_0)$, $\lambda:={1/f'(\infty)>1}$, for all~${n\in\Natural}$ and a suitable ${\zeta_0\in\UH}$.

    \StepP{\ref{IT_THmain3:1}}
			Since $F$ has a semiregular fixed point at $\infty$, Lemma~\ref{LM_isogonality}\,\ref{IT_LM_isogonality(A)} implies that there exists some $n_0 \in \mathbb{N}_0$ large enough such that ${F^{\prime}(\lambda^n \zeta_0 )\neq 0}$ for all ${n \geq n_0}$. Substituting  ${\zeta:=\lambda^{n_0+m} \zeta_0}$ in~\eqref{EQ_premodel-der}, we have that ${f'(w_{n_0+m})}$ and ${F'(\lambda^{n_0+m} \zeta_0)}$ are both different from zero for all~$m\geq 1$.

			\StepP{\ref{IT_THmain3:2}}	Combining~\eqref{EQ_premodel-eq-in-the-proof} with \eqref{EQ_premodel-der}, we obtain
			\begin{equation}\label{EQ_hyperbolic-derivative-of-the-iterates}
				\dhyp f^{\circ{k}}\big(F(\zeta)\big) \dhyp F(\zeta)\,=\,\dhyp F(\lambda^{-k}\zeta) \, , \quad \zeta\in\UH \, , \, k\in\Natural\,.
			\end{equation}

   As we have concluded in the proof of~\ref{IT_THmain3:1}, ${F'(\lambda^{n_0+m} \zeta_0)\neq0}$ and, as a consequence, also ${\dhyp F(\lambda^{n_0+m} \zeta_0)\neq0}$ for all ${m\in\Natural}$.
			Substituting ${k:=m}$ and ${\zeta:=\lambda^{n_0+m} \zeta_0}$ in~\eqref{EQ_hyperbolic-derivative-of-the-iterates}, we therefore obtain
			\begin{equation}\label{EQ_f<->F}
				\dhyp f^{\circ m}(w_{n_0+m})=\frac{\dhyp F(\lambda^{n_0}\zeta_0)}{\dhyp F(\lambda^{n_0+m} \zeta_0)}.
			\end{equation}
    Theorem~\ref{TH-main1} implies that $\dhyp F(\zeta) \to 1$ as $\zeta\to \infty$ non-tangentially. In particular,
				\begin{equation}\label{EQ_Dh-to-1-along-an-orbit}
					\dhyp F(\lambda^{n_0+m} \zeta_0)\to 1 \quad \text{as~$~{m\to+\infty}$}.
				\end{equation}
				By~\eqref{EQ_f<->F} and~\eqref{EQ_Dh-to-1-along-an-orbit}, the limit $\mu:=\lim_{m\to+\infty}f^{\circ m}(w_{n_0+m})$ exists and equals~$\dhyp F(\lambda^{n_0}\zeta_0)$. By the very definition of~$\dhyp F$, we have ${\dhyp F(\lambda^{n_0}\zeta_0)\ge0}$. Moreover, by the choice of~$n_0$ made in the proof of~\ref{IT_THmain3:1},  ${\dhyp F(\lambda^{n_0}\zeta_0)\neq0}$. In addition, by the invariant Schwarz\,--\,Pick Lemma, ${\dhyp F(\lambda^{n_0}\zeta_0)\le1}$. Therefore, ${\mu\in(0,1]}$ as desired.

			\StepP{\ref{IT_THmain3:3}}
			Using~\eqref{EQ_f<->F}, we have
			\begin{equation}\label{EQ_two-series}
				\sum_{m=1}^{+\infty}\big(\dhyp f^{\circ m}(w_{n_0+m})-\mu\big)~=~
				\mu\,\sum_{m=1}^{+\infty}\Big(\,\frac1{\dhyp F(\lambda^{n_0+m} \zeta_0)}-1\Big).
			\end{equation}
In view of~\eqref{EQ_Dh-to-1-along-an-orbit}, convergence of the series in the r.h.s. is equivalent to~\eqref{EQ_main-cond-discrete_form}. By Proposition~\ref{PR_cont<=>discrete}, the latter is in turn equivalent to~\eqref{EQ_int-conv-cond-weak}. Finally, recalling that $F$ has a semiregular fixed point at~$\infty$, we have that according to Proposition~\ref{PR_weak} and Lemma~\ref{LM_elementary-implication}, the boundary fixed point of~$F$ at~$\infty$ is regular if and only if~\eqref{EQ_int-conv-cond-weak} holds. Thus, the pre-model for~$f$ at~$\infty$ is regular if and only if the series in the l.h.s. of~\eqref{EQ_two-series} converges. 
		\end{proof}

  		\begin{proof}[\proofof{Corollary~\ref{CR_fromTH-main3}}]
			Let $(w_n)$ be a regular backward orbit converging to the repulsive fixed point at~$\infty$. We assume without loss of generality that ${n_0=0}$ in~\ref{IT_THmain3:1} and thus, $f'(\zeta_n)\neq0$ for all~${n\in\Natural}$. Then also ${\dhyp f^{\circ n}(w_n)\neq0}$ for all~${n\in\Natural}$.

			According to Theorem~\ref{TH-main3}, the pre-model of~$f$ at~$\infty$ is regular if and only if
			\begin{equation}\label{EQ_discrete-cond-copy}
				\sum_{n=1}^{+\infty}\big(\dhyp f^{\circ n}(w_n)-\mu\big)~<~+\,\infty,\quad \mu:=\lim_{n\to+\infty}\dhyp f^{\circ n}(w_n)\,\in\,(0,1].
			\end{equation}
			Therefore, it suffices to show that convergence of the series~\eqref{EQ_discrete-cond-copy} is equivalent to the convergence of the following series
			\begin{equation}\label{EQ_series-in-the-corollary}
				\sum_{n=1}^{+\infty}\big(\rho-\hypdist_\UH(w_n,w_{n+1})\big),
			\end{equation}
			where $\rho:=\lim_{n\to+\infty}\hypdist_\UH(w_n,w_{n+1})$. This limit exists by Remark~\ref{RM_step}, and it is finite because the backward orbit is regular by the hypothesis. Again by Remark~\ref{RM_step}, we have
			${\hypdist_\UH(w_n,w_{n+1})\ge\hypdist_\UH(w_1,w_2)>0}$, where the second inequality sign holds, since ${w_2\neq f(w_2)=w_1}$. It then follows from Proposition~\ref{prop:Ineq_ApprUniv} that there exist two constants ${0<C_1<C_2<+\infty}$ such that for any ${n,m\in\Natural}$ with ${m>n}$, we have
			\begin{equation}\label{EQ_through-dhyp}
				C_1\big(1-\dhyp f^{\circ(m-n)}(w_m)\big)
				\le \hypdist_\UH(w_{m},w_{m+1})-\hypdist_\UH(w_{n},w_{n+1}) \le
				C_2\big(1-\dhyp f^{\circ(m-n)}(w_m)\big).
			\end{equation}
			Note that $$\dhyp f^{\circ(m-n)}(w_m)=\frac{\dhyp f^{\circ m}(w_m)}{\dhyp f^{\circ n}(w_n)}~\to~\frac{\mu}{\dhyp f^{\circ n}(w_n)}\quad\text{as~$~m\to+\infty$}.$$

			Therefore, passing in~\eqref{EQ_through-dhyp} to the limit as ${m\to+\infty}$, we conclude that convergence of the series~\eqref{EQ_series-in-the-corollary} is equivalent to that of $\sum_{n=1}^{+\infty}\big(1-\mu/\mu_n\big)$, where $\mu_n:=\dhyp f^{\circ n}(w_n)$. As desired, the latter is in turn equivalent to~\eqref{EQ_discrete-cond-copy}, because ${\mu_n\to\mu}$ as ${n\to+\infty}$.

            It remains to mention that the relation
				$$
				\rho=2\,\arctanh\frac{1-f'(\infty)}{\big|1+e^{2i\theta}f'(\infty)\big|},\quad \theta:={\lim_{n\to+\infty}\arg w_n},
				$$
				is known, see e.g. \cite[Corollary~4.9.13]{Marco}. Alternatively, it follows from the half-plane version of \ref{IT_TH-main1:semireg-cntct-pnt}$\,\Rightarrow\,$\ref{IT_Th-main1:hdiffquo-to-1} in Theorem~\ref{TH-main1} applied to~$F$.
		\end{proof}

		\section{Concluding remarks} \label{SEC:ConcludingRemarks}

		\addtocontents{toc}{\SkipTocEntry}\subsection{A word about terminology}\mbox{~}
		\label{SUB:terminology}
			The two notions of boundary conformality for holomorphic self-maps, which we consider in this paper, have counterparts for conformal mappings of~$\UD$. Unfortunately, there is no common consent in the literature regarding terminology. In the terminology introduced by Pommerenke in~\cite{Pommerenke:BB}, which we prefer to follow,  ${f:\UD\to\C}$ is said to be \emph{conformal} at~$\sigma\in\UC$ if the angular limit $f(\sigma):=\anglim_{z\to\sigma} f(z)$ exists finitely and $f$ has at~$\sigma$ finite non-vanishing angular derivative $f'(\sigma):={\anglim_{z\to\sigma}\big(f(z)-f(\sigma)\big)/(z-\sigma)}$. In the earlier monograph~\cite{Pombook75} by the same author, as well as in some other sources, e.g. \cite[\S{}V.5]{GarnettMarshall2005}, the term ``conformality'' is used for another, a priori and a posteriori weaker property which, instead of asking for finite non-vanishing angular derivative, is analytically described by the existence of the finite angular limit of ${\arg\big(f(z)-f(\sigma)\big)/(z-\sigma)}$. Geometrically, this is equivalent to a certain angle preservation property. By this reason, following~\cite{Pommerenke:BB}, we refer to this weaker property as \emph{isogonality}. Another term, which appears often in the literature (e.g. \cite{BCD-Book,JO1977}) in the same meaning as isogonality, is \emph{semi-conformality}.

			Clearly, both conformality and isogonality can be extended to the case of infinite angular limit ${f(\sigma)=\infty}$ by replacing $f$ with~$1/f$. Similarly, these properties can be reformulated in the ``strip setting'', i.e. for conformal mappings of an infinite strip at one of its boundary points at~$\infty$; see e.g. \cite{BK2022,RodinWarsch,JO1977}.
			Moreover, Yamashita~\cite{Yamashita} studied isogonality for \textit{non-injective} holomorphic functions in~$\UD$, using the term ``conformality'' instead. In the same paper, the Visser\,--\,Ostrowski condition (see e.g. \cite[\S11.3]{Pommerenke:BB}) is also investigated and referred to as ``semi-conformality''. Note that in general, the latter condition is much weaker than isogonality, but they become equivalent if $f$ is a holomorphic self-map and the value of the angular limit at the boundary point we consider belongs again to the boundary; this follows, for example, from our Lemma~\ref{LM_isogonality}.

			By the very definition (see Section~\ref{SEC:MainResults}) for a holomorphic self-map~$\varphi\in\Hol(\UD,\UD)$, the conformality at~$\sigma\in\UC$ in the weak (resp., strong) sense is equivalent to isogonality (resp., conformality) at~$\sigma$, combined with the requirement that $\varphi(\sigma)\in\UC$.

		\addtocontents{toc}{\SkipTocEntry}
		\subsection{Theorem \ref{TH-main1}, Theorem \ref{TH-main2} and  boundary Schwarz lemmas of Burns-Krantz--type}
		As noted in the Introduction, Theorems \ref{TH-main1} and \ref{TH-main2} can be seen as local rigidity theorems in the spirit of the boundary Schwarz lemma of Burns and Krantz \cite{BurnsKrantz1994}. Indeed, for the special case of the unit disk the Burns-Krantz Schwarz lemma asserts that if $\varphi \in \Hol(\UD,\UD)$ and $\sigma \in \partial \D$, then
		\def\bliz{\hspace{-.5em}}
		\begin{align} \label{EQ:BK}
			\varphi(z)&=z+o\big(|1-z|^3|\big) \quad \text{as~}~z \to \sigma~\text{~non-tangentially }    & \bliz\Longleftrightarrow\quad  &\varphi=\text{id}  \, .\\
			\intertext{Recently, the following  invariant version  of the Burns-Krantz lemma has been established in \cite[Theorem 2.1]{FDO}:}
			\label{EQ:FDO} \dhyp \varphi(z)&=1+o\big(|1-z|^2\big)  \quad \text{as~}~z \to \sigma~\text{~non-tangentially }    & \bliz\Longleftrightarrow\quad    &\varphi \in \Aut(\UD) \, .\\
			\intertext{In fact, it is not difficult to show (see \cite[Remark 2.2]{FDO}) that \eqref{EQ:FDO} implies \eqref{EQ:BK}. It is a natural question to wonder which properties of $\varphi \in \Aut(\UD)$ continue to hold if the condition on the left-hand side of \eqref{EQ:FDO} is weakened. In this direction, Kraus, Ruscheweyh and the last named author have proved in \cite{KRR2007} that the condition}
			\label{EQ:DSO} \dhyp \varphi(z)& =1 +o\big(1 \big)   \quad \text{as~}~z \to \sigma~\text{~unrestrictedly }\\
\intertext{ is equivalent to the property that $\varphi$ can be analytically continued by Schwarz reflection across an open subarc $J\subseteq \partial \D$ containing $\sigma$.
Theorem~\ref{TH-main1} yields the additional information that}
			\label{EQ:PMAO} \dhyp \varphi(z)&=1+o\big(1\big)     \quad\text{as~}~z \to \sigma~\text{~non-tangentially }    \mathrlap{\quad \Longleftrightarrow \quad  \varphi \text{ is weakly conformal at } \sigma.} &   &\hspace{8em}\mbox{~}
		\end{align}
Furthermore, Theorem~\ref{TH-main2} tells that the residual term in~\eqref{EQ:PMAO} is integrable along the radial segment ${[0,\sigma)}$ w.r.t. hyperbolic arc length if and only if $\varphi$ is \emph{strongly} conformal at~$\sigma$.

		%*************************************************************************
		\addtocontents{toc}{\SkipTocEntry}
		\subsection{Theorem \ref{TH-main2} and the angular derivative problem for conformal maps}
		\label{SUB:AngDerivativeConformalMaps}
		The notion of conformality at the boundary has mostly been considered for univalent mappings, and often in connection with the famous \emph{angular derivative problem}. This problem  goes back to the thesis of Ahlfors~\cite{Ahlfors1930} and
		consists in finding necessary and sufficient geometric conditions on a simply connected subdomain $G$ of $\C$ near a boundary point ${\omega\in \partial G}$, so that one --- and hence any --- univalent map~$\varphi$ from $\UD$ onto~$G$ admits, at the point ${\sigma\in\UC}$ corresponding\footnote{To be precise, $\partial G$ should be understood in this case as the Carath\'eodory boundary of the domain~$G$ rather than its topological boundary, and $\omega$ should be a prime end containing an accessible point. The latter guarantees the existence of the angular limit $\varphi(\sigma):=\anglim_{z\to\sigma}\varphi(\sigma)$; see, e.g., \cite[\S9.2 and \S10.2]{Pombook75} for more detail.} under~$\varphi$ to~$\omega$, finite non-vanishing angular derivative $\varphi'(\sigma):={\anglim_{z\to\sigma}\big(f(z)-f(\sigma)\big)/(z-\sigma)}$.
This problem has attracted much interest, and continues to do so. Seminal contributions have been given by Rodin and Warschawski \cite{RW-FinnishJ, RodinWarsch}, Jenkins and Oikawa \cite{JO1977}, followed by many others, e.g. Burdzy~\cite{Burdzy1986}, Carroll~\cite{Carroll1988}, as well as very recently by Betsakos and Karamanlis~\cite{BK2022}.

It is worth mentioning that \cite[Theorem~1]{BK2022}, being suitably translated to the unit disk setting and combined with some standard arguments, characterizes existence of a finite non-vanishing angular derivative $\varphi'(\sigma)$ for a properly normalized conformal mapping~$\varphi:D\xrightarrow{\text{onto}}\UD$, where $D$ is a simply connected domain sharing with the unit disk the boundary point~$\sigma$ and possessing at~$\sigma$ a sort of one-sided isogonality property\footnote{The exact meaning of ``one-sided isogonality'' is that for any Stolz angle $S$ with vertex at~$\sigma$ there is ${\varepsilon>0}$ such that ${\{z\in S\colon |z-\sigma|<\varepsilon\}\subset D}$. We refer to this condition as ``one-sided'' because it is \emph{not} supposed that ${D\subset\UD}$.}.
			This interpretation of \cite[Theorem~1]{BK2022} yields literally the same necessary and sufficient condition as our condition~\ref{IT_TH-main2:BetsakosKaramanlis-radial} in the Addendum~\ref{ADD:TH-main2Condition(d)}. However, it is clear that none of the results follows from the other, as the assumptions on the map~$\varphi$ are rather different.

		%*************************************************************************

		\addtocontents{toc}{\SkipTocEntry}
		\subsection{Angular derivative vs.~non-tangential limit of hyperbolic distortion}
		\label{SUB:AngDerVSHypDist}
Condition \eqref{EQ_anglim-abs-hder=1}, i.e.,
\begin{equation}\label{EQ_anglim-abs-hder=2}
				\anglim_{z\to\sigma} \dhyp \varphi(z)~=~1\, ,
\end{equation}
which we identify in Theorem~\ref{TH-main1} as being equivalent  to the conformality of $\varphi \in \Hol(\UD,\UD)$ at~$\sigma$  \emph{in the weak sense}, has recurrently been considered in the literature in   connection with existence of the angular derivative.

		For instance, Beardon and Minda \cite[Theorem 7 $+$ Theorem 8]{BM2023} very recently proved that \eqref{EQ_anglim-abs-hder=2} as well as condition~\ref{IT_Th-main1:hdiffquo-to-1} in Theorem~\ref{TH-main1} hold provided that $\varphi$ is conformal at ${\sigma \in \partial \D}$ \emph{in the strong sense}. This was also observed before in \cite[Lemma~2.9\,(b)]{Kraus2013}. We also note that Lemma~2.9\,(a) in~\cite{Kraus2013} asserts that if $\angle \lim_{z \to \sigma} \dhyp \varphi(z)=1$ holds for \textbf{every} ${\sigma \in \partial \D}$, then $\varphi$ is conformal in the strong sense at a.e.~point of~$\partial \D$. This result easily follows from Theorem~\ref{TH-main1} above together with~\cite[Theorem~3]{Yamashita}, which says that if $\varphi$ is isogonal everywhere on~$\partial \D$, then~$\varphi$ has a finite angular derivative at a.e.~$\sigma~\in~\partial \D$. Another relation (of partially quantitative nature) between strong conformality  a.e. on~$\UC$ and boundary behaviour of~$\dhyp\varphi$ is contained in the main result of the recent preprint~\cite{IvriiNicolau}. 

		It has been known for a while that condition \eqref{EQ_anglim-abs-hder=2}
		in general does not imply that $\varphi$ has a finite angular derivative at~$\sigma$.
		An example is furnished  in \cite{Zorboska2015}. Theorem \ref{TH-main1} shows that the explicit function $\varphi \in \Hol(\UD,\UD)$ described in \cite{Zorboska2015} is conformal at $\sigma$ in the weak sense.

		\addtocontents{toc}{\SkipTocEntry}
		\subsection{Metric aspects of Theorem \ref{TH-main1} and Theorem \ref{TH-main2}} \label{SUB:MetricAspects}
                    The conditions (a) and (b) in Theorem~\ref{TH-main1} as well as in Theorem~\ref{TH-main2} can be expressed purely in terms of the hyperbolic distance $\hypdist_\D$ and the induced hyperbolic distortion
                    $$ \dhyp \varphi(z):=\lim \limits_{w \to z} \frac{\hypdist_\D\big(\varphi(z),\varphi(w)\big)}{\hypdist_\D\big(z,w\big)} \, .$$
                    It is  a natural question to wonder to what extend our results generalize to the case of  non-expansive self-maps $\Phi :X \to X$ of a metric space $(X,\hypdist_X)$. In this more general setting, the analogue of the hyperbolic distortion seems to be the lower distortion
                    $$ \liminf \limits_{y \to x} \frac{\hypdist_X\big(\Phi(x),\Phi(y)\big)}{\hypdist_X \big(x,y\big)} \, . $$
                   In this regard, it is also worth recalling that the strong conformality condition~\ref{IT_TH-main2:regular-cntct-pnt} in Theorem~\ref{TH-main2} is equivalent to the purely metric condition~\eqref{EQ_Julia-metric}, which can be used as its substitute when the requirement of holomorphicity is relaxed. 

 For recent works extending the circle of ideas around  hyperbolic distortion, conformality at the boundary in the strong sense~--- and their relation with (forward and backward) dynamics of the iterates~--- to non-expansive self-maps of the unit ball in $\C^n$ or more general metric spaces,  we refer to e.g.~\cite{AFGG2024, AFGK2024, Beardon:dynamics, BM2023, BGZ2021, HW2021, Karlsson2001, LN2008} and references therein.

                \section*{Acknowledgements}
                The authors wish to thank O.~Ivrii for sending a preprint \cite{Ivr} of his joint work with M.~Urba\'{n}ski. They also thank A.~Nicolau for useful discussions and sending a preprint \cite{IvriiNicolau} of his recent work with O.~Ivrii which deals with the related problem of studying  holomorphic self-maps which  almost preserve hyperbolic area.  The first author is grateful to L.~Arosio for fruitful discussions related to the applications to backward dynamics.

		\renewcommand\appendixname{}
		\appendix
		\section{Appendix}

Here we present some auxiliary results related to conformality in the weak sense that were used in the proofs of previous sections. We again work in the half-plane setting, see Section~\ref{SUB:half-plane}.
Recall the notation $~A_{\theta}:=\big\{\zeta\colon|\arg\zeta|<\theta\}\subset\UH~$ and
			$$
        A_{\theta,R}:=\big\{\zeta\colon|\arg\zeta|<\theta,~|\zeta|>R\big\}\subset\UH,\qquad R>0,~\theta\in(0,\pi/2).
			$$

\begin{lemma}\phantomsection\label{LM_isogonality}
			\begin{romlist}
      \item[]
      \item\label{IT_LM_isogonality(A)} Suppose $F\in\Hol(\UH,\UH)$ has a boundary fixed point at~$\infty$. Then the following two conditions are equivalent:
		\begin{arablist}
			\item\label{IT_LM_isogonality-deri}
		     For any $\theta\in(0,\pi/2)$ there exists $R>0$ such that $F'$ does not vanish in $A_{\theta,R}$ and
					$$
					\frac{F'(\zeta)}{|F'(\zeta)|}\to1\quad\text{as $~A_{\theta,R}\ni \zeta\to\infty$.}
					$$
			\item\label{IT_LM_isogonality-isog} there exists the angular limit
					$$
					\anglim_{\zeta\to\infty}\arg\frac{F(\zeta)}{\zeta}=0.
					$$
		\end{arablist}
        Moreover, if the above equivalent conditions hold, then
				\begin{equation}\label{EQ_LM_isogonality}
					\anglim_{\zeta\to\infty}\frac{\zeta F'(\zeta)}{F(\zeta)}=1.
				\end{equation}

		\item\label{IT_LM_isogonality(B)}
       Conversely, if a self-map $F\in\Hol(\UH,\UH)$ satisfies~\eqref{EQ_LM_isogonality}, then $F$ has a boundary fixed point at~$\infty$ and the above conditions~\ref{IT_LM_isogonality-deri} and~\ref{IT_LM_isogonality-isog} hold.
			\end{romlist}
\end{lemma}

A couple of observations are worth being made before we prove the above lemma.
			\begin{remark}\label{RM_without-limit-values}
				It is easy to see that for any $F\in\Hol(\UH,\UH)$, condition~\ref{IT_LM_isogonality-isog} in the above lemma is equivalent to an \textit{a priori} weaker condition that $\alpha:=\anglim_{\zeta\to\infty}\arg\big(F(\zeta)/\zeta\big)$ exists but does not necessarily equal~$0$. Indeed, if this limit exists then ${\alpha=0}$, because otherwise setting $\beta:={\mathrm{sgn}(\alpha)\big(\pi-|\alpha|\big)/2}$, we would get $\big|\arg F(Re^{i\beta})\big|\to\big(\pi+|\alpha|\big)/2$ as ${R\to+\infty}$, which would imply ${F(\UH)\not\subset\UH}$.

				Similarly, under the hypothesis of Lemma~\ref{LM_isogonality}, condition~\ref{IT_LM_isogonality-deri}  is equivalent to the mere existence of ${\kappa:=\anglim_{\zeta\to\infty}F'(\zeta)/|F'(\zeta)|}$. Indeed, as the proof of Lemma~\ref{LM_isogonality} given below shows, if the latter limit exists, then $\alpha:=\anglim_{\zeta\to\infty}\arg\big(F(\zeta)/\zeta\big)$ exists as well and ${\kappa=e^{i\alpha}}$. As we have just seen, ${\alpha=0}$ because ${F(\UH)\subset\UH}$. Consequently, we necessarily have ${\kappa=1}$.
			\end{remark}

			\begin{remark}\label{RM_Yamashita}
				Taking into account the previous remark and the correspondence between holomorphic self-maps of $\UD$ and those of~$\UH$, see Section~\ref{SUB:half-plane}, it is not difficult to see that the equivalence between conditions \ref{IT_LM_isogonality-deri} and~\ref{IT_LM_isogonality-isog} in Lemma~\ref{LM_isogonality} implies the equivalence of condition~\ref{IT_TH-main1:semireg-cntct-pnt} in Theorem~\ref{TH-main1} to condition~\ref{IT_FEQ_b} in Addendum~\ref{ADD:TH-main1}. Alternatively, this can be deduced from a result of Yamashita~\cite[Theorem~2]{Yamashita}. We however prefer to give a simpler and direct proof, which in addition allows us to show that \ref{IT_LM_isogonality-deri} and~\ref{IT_LM_isogonality-isog} are further equivalent, under the hypothesis of Lemma~\ref{LM_isogonality}\,\ref{IT_LM_isogonality(A)}, to the half-plane version~\eqref{EQ_LM_isogonality} of the Visser\,--\,Ostrowski condition~\eqref{EQ:VisserOstrowski}.  The latter fact is used in our proof of Theorem~\ref{TH-main1}, and it does not hold in the general case considered by Yamashita as mentioned e.g. in \cite[pp.\,251-252]{Pommerenke:BB}.
			\end{remark}

			In view of condition~\ref{IT_LM_isogonality-deri} in Lemma~\ref{LM_isogonality}, it is natural to introduce angular limits at~$\infty$ for functions defined on subdomains of~$\UH$, in general, different from~$\UH$.
			\begin{definition}\label{DF_domain-isogonality}
				A  domain $U\subset \UH$ is said to be \textit{isogonal} at~$\infty$, if for every~${\theta\in(0,\pi/2)}$ there exists ${R=R(\theta)>0}$ such that ${A_{\theta,R}\subset U}$.
			\end{definition}
			\begin{definition}\label{DF_anglim-in-iso-domain}
				Let $f$ be a complex-valued function defined in a domain~$U\subset\UH$. We will say that $f$ has \emph{angular limit~${\ell\in\overline{\C}}$ at~$\infty$} and write ${\anglim_{\zeta\to\infty}f(\zeta)=\ell}$, if~$U$~is isogonal at~$\infty$ and
				$$
				f(\zeta)\to\ell\quad\text{as~$~U\cap A_\theta\ni\zeta\to\infty$,}
				$$
				for any $\theta\in(0,\pi/2).$
			\end{definition}
\begin{proof}[\proofof{Lemma~\ref{LM_isogonality}\,\ref{IT_LM_isogonality(A)}}]\mbox{~}

\StepP{\ref{IT_LM_isogonality-deri}$\Longrightarrow$\ref{IT_LM_isogonality-isog}}
			Fix some $\theta\in(0,\pi/2)$ and $\varepsilon\in(0,\pi)$. Then by~\ref{IT_LM_isogonality-deri}, there exists ${R>0}$ such that $F'(\zeta)\in A_{\varepsilon/2}$ for all ${\zeta\in A_{\theta,R}}$. Since for any ${\zeta\in A_{\theta,R}}$ the Euclidean segment joining $R\zeta/|\zeta|$ to~$\zeta$ lies (except for one of the end-points) in $A_{\theta,R}$ and since $A_{\varepsilon/2}$ is convex, it follows that
			$$
			Q_R(\zeta):=\frac{F(\zeta)-F(R\zeta/|\zeta|)}{\zeta}=\frac{|\zeta|-R}{|\zeta|}\int_{0}^{1} F'\big(\omega(t)\big)\di t~\in~A_{\varepsilon/2},\quad  \omega(t):=\frac{R+(|\zeta|-R)t}{|\zeta|}\zeta.
			$$
As a consequence,
			$$
			\left|\arg \frac{F(\zeta)}{\zeta}\right|~\le~\big|\arg Q_R(\zeta)\big|\,+\,\left|\arg\left(1-\frac{F(R\zeta/|\zeta|)}{F(\zeta)}\right)\right|
			~<~\frac{\varepsilon}2\,+\,\arcsin\frac{M_R}{|F(\zeta)|},
			$$
			where $M_R:=\max_{|t|\le\theta}|F(Re^{it})|$.

			Since by the hypothesis, $\anglim_{\zeta\to\infty}F(\zeta)=\infty$, it follows that there exists ${R_1>0}$ such that ${|F(\zeta)|>M_R/\sin(\varepsilon/2)}$  for all ${\zeta\in A_{\theta,R_1}}$. Further, it follows that for all $\zeta\in A_{\theta,\rho}$, where ${\rho:=\max\{R,R_1\}}$, we have ${\big|\arg(F(\zeta)/\zeta)\big|<\varepsilon}$.
			This proves~\ref{IT_LM_isogonality-isog} since $\theta\in(0,\pi/2)$ and $\varepsilon\in(0,\pi)$ are arbitrarily chosen.

\StepP{\ref{IT_LM_isogonality-isog}$\Longrightarrow$\ref{IT_LM_isogonality-deri} and \eqref{EQ_LM_isogonality}} Using the conformal map ${w\mapsto e^w}$ of ${\mathbb S:=\{w \in \C :|\Im w|<\pi/2\}}$ onto~$\UH$, we obtain ${f\in\Hol(\mathbb S,\mathbb S)}$ such that $F(e^w)=\exp f(w)$ for all ${w\in\mathbb S}$. Denote $g(w):={f(w)-w}$, ${w\in\mathbb S}$.
			Condition~\ref{IT_LM_isogonality-isog} implies that
			\begin{equation}\label{EQ_Im-g-0}
				\anglim_{w\to+\infty} \Im g(w)=0.
			\end{equation}
			Here, the angular limit of $g:\mathbb S\to\C$ as ${w\to+\infty}$ is defined as ${\ell\in\overline\C}$ such that for any ${\delta\in(0,\pi/2)}$,
			$$
			g(w)\to\ell\quad\text{~as}\quad\mathbb S^+_\delta\ni w\to\infty, \qquad \text{where~
				$~\mathbb S_\delta^+:=\{w:\Re w>0,~|\Im w|<\pi/2-\delta\}$.}
			$$

			Condition~\eqref{EQ_LM_isogonality} can be written as $\anglim_{w\to+\infty} f'(w)=1$. Moreover, note that \eqref{EQ_LM_isogonality} combined with~\ref{IT_LM_isogonality-isog} would immediately imply~\ref{IT_LM_isogonality-deri}.
			Therefore, to prove~\ref{IT_LM_isogonality-deri} and~\eqref{EQ_LM_isogonality}, we have to show that
			\begin{equation}\label{EQ_g-prime}
				\anglim_{w\to+\infty} g'(w)=0.
			\end{equation}
			To this end, we fix some arbitrary ${\delta\in(0,\pi/2)}$ and ${\varepsilon>0}$. It follows from~\eqref{EQ_Im-g-0} that there exists~${u_0>0}$ such that if \def\uppest{\varepsilon\delta/8}$w\in\mathbb S_{\delta/2}^+$ and ${\Re w>u_0}$, then $|\Im g(w)|<\uppest=:h(\delta,\varepsilon)$. Therefore, the function $g$ maps the half-strip $U_{\delta,\varepsilon}:=\{{w\in\mathbb S_{\delta/2}^+}:{\Re w>u_0}\}$ into the strip $S_{\delta,\varepsilon}:=\{\zeta:{|\Im\zeta|<h(\delta,\varepsilon)}\}$. Thus, for any $w\in\mathbb S_{\delta}^+$ with ${\Re w>u_0+\delta/2}$, we have
			$$
			|g'(w)|\le \frac{\lambda_{U_{\delta,\varepsilon}}(w)}{\lambda_{S_{\delta,\varepsilon}}(g(w))} < 4\,\frac{h(\delta,\varepsilon)}{\delta/2}=\varepsilon,
			$$
			where we used the general Schwarz\,--\,Pick Lemma and the classical inequality relating the hyperbolic density $\lambda_D$ to the distance to~$\partial D$, where $D$ is a simply connected domain, see e.g. \cite[Theorem~6.4, ineq.\,(8.4)]{BM2007}.
			Since ${\delta\in(0,\pi/2)}$ and ${\varepsilon>0}$ are arbitrary, this proves~\eqref{EQ_g-prime}.
			\phantom\qedhere
		\end{proof}

		\begin{proof}[\proofof{Lemma~\ref{LM_isogonality}\,\ref{IT_LM_isogonality(B)}}] If $F$ satisfies~\eqref{EQ_LM_isogonality}, then
			$$
			\log |F(x)|=\Re \log F(x)=\Re \log F(1)+\int \limits_{1}^x \Re \left(\frac{\xi F'(\xi)}{F(\xi)}\right) \frac{\di\xi}{\xi} \to +\infty \quad\text{as~$x~\to+\infty$}, $$
			since $\Re (\xi F'(\xi)/F(\xi)) \to 1$. Hence $|F(x)| \to+\infty$ as $x \to+\infty$, so by the Lindel\"of angular limit theorem, see e.g. \cite[Theorem~9.3]{Pombook75} or \cite[Theorem 1.5.7]{BCD-Book}, ${F(z)\to\infty}$ as ${z \to \infty}$ non-tangentially, i.e. $F$ has a boundary fixed point at $\infty$. To show that~\ref{IT_LM_isogonality-isog} and hence~\ref{IT_LM_isogonality-deri} hold, we again pass to the map ${f\in\Hol(\mathbb S,\mathbb S)}$ defined in the proof of~\ref{IT_LM_isogonality(A)} and adopt the notation introduced there. Condition~\eqref{EQ_LM_isogonality} means that $\anglim_{w\to+\infty}f'(w)=1$.  Using the invariant form of the Schwarz\,--\,Pick Lemma and the explicit formula for $\lambda_{\mathbb S}$, we therefore have
			$$
			\cos\big(\Im f(x)\big)=\frac{\lambda_{\mathbb S}\big(x\big)}{\lambda_{\mathbb S}\big(f(x)\big)} \ge |f'(x)| \to 1\quad\text{as~$~\Real\ni x\to+\infty$.}
			$$
			It follows that $\Im f(x)\to0$ as~$~\Real\ni x\to+\infty$.
			Moreover,
			$$
			\Im\big(f(w)-w\big)=\Im f(\Re w)+\Im \int\limits_{\Re w}^w\! \big( f'(\zeta)-1 \big) \, \di\zeta\quad\text{for any~$w\in\mathbb S$.}
			$$
			Bounding the integral above by $|\Im w|\cdot \max\big\{|f'(\zeta)-1|\colon \zeta \in [\Re w,\,w]\big\}$, we conclude  that $\anglim_{w\to+\infty}\Im\big(f(w)-w\big)=0$, which proves~\ref{IT_LM_isogonality-isog}.
		\end{proof}

There is a one-to-one relation between the isogonality of a simply connected domain and isogonality of the corresponding Riemann map, see e.g.~\cite[\S11.3]{Pommerenke:BB} or~\cite[\S{}V.5]{GarnettMarshall2005}. Moreover, for holomorphic self-maps, not necessarily univalent, the following result holds, which we believe is known to specialists and which we prove here because of lack of a suitable reference for the non-univalent case.

		\begin{proposition}\label{PR_known1}
			Suppose that $F\in\Hol(\UH,\UH)$ has a boundary  fixed point at~$\infty$ and that it is conformal in the weak sense at~$\infty$.  Then there exists a simply connected domain~$U\subseteq\UH$ satisfying the following conditions:
			\begin{romlist}
				\item\label{IT_known1-isogonal} $U$ is isogonal at~$\infty$.
				\item\label{IT_known1-univalent} $F$ is univalent in~$U$.
				\item\label{IT_known1-image-isogonal} $W:=F(U)$ is also isogonal at~$\infty$.
				\item\label{IT_known1-F_inverse-isogonal} The inverse function $G:=(F|_U)^{-1}:W\to U$  satisfies
					$$
					\anglim_{\zeta\to\infty}G(\zeta)=\infty,\quad~  \anglim_{\zeta\to\infty}\arg\frac{G(\zeta)}{\zeta}=0,\quad \text{~and}\quad \anglim_{\zeta\to\infty}\frac{G'(\zeta)}{|G'(\zeta)|}=1.
					$$
			\end{romlist}
		\end{proposition}
The angular limits in~\ref{IT_known1-F_inverse-isogonal} above are to be understood in the sense of Definition~\ref{DF_anglim-in-iso-domain}.
		\begin{proof}
\StepG{Construction of a domain~$U$ satisfying~{\rm\ref{IT_known1-isogonal}}.\,} For $n\in\Natural$, denote $\theta_n:=\pi/2-1/2^n$. According to Lemma~\ref{LM_isogonality}\,\ref{IT_LM_isogonality(A)}, for each ${n\in\Natural}$ there exists $R_n>0$ such that
			\begin{equation}\label{EQ_F-prime}
				F'(\zeta)\neq0\quad\text{and}\quad \big|\arg F'(\zeta)\big|<1/2^{n+1} \quad \text{for all~}~\zeta\in \C\cap\overline{A_{\theta_n,R_n}}.
			\end{equation}
			Note that
			$$
			U_n:=\big\{\zeta\in\UH\colon|\arg \zeta|<\theta_n\big\}\,\bigcap\,\big\{\zeta\in\UH\colon \Re \zeta>R_n,~|\arg(\zeta-R_n)|<\theta_{n+1}\big\}~\subseteq~A_{\theta_n,R_n}.
			$$
			Clearly, the union $U:=\bigcup_{n\in\Natural}U_n$ is a domain of the form $\{u+iv:u>R_1,~|v|<f(u)\}$, where ${f:[R_1,+\infty)\to[0,+\infty)}$ is an increasing continuous function with ${f(R_1)=0}$. Therefore, $U$ is simply connected. Moreover,~\ref{IT_known1-isogonal} holds since ${U_n\supset A_{\theta_n,\rho_n}}$ for any ${n\in\Natural}$, where $\rho_n:={R_n\sin\theta_{n+1}/\sin(\theta_{n+1}-\theta_n)}$.

			\StepP{\ref{IT_known1-univalent}} Fix two distinct points $\zeta_1,\zeta_2\in U$ and choose a pair of integers ${n\ge1}$ and ${m\ge n}$ such that ${\zeta_1,\zeta_2\in U_{n,m}}:={U_n\cup U_m}$.
			By construction,
			$$
			U_{n,m}=\big\{u+iv\colon u>R_n,~|v|<f_{n,m}(u) \big\},
			$$
			where $f_{n,m}$ is an increasing function in $[R_n,+\infty)$ with $f_{n,m}(R_n)=0$. Moreover,
			\begin{equation}\label{EQ_fnm-Lip}
				\big|f_{n,m}(u_2)-f_{n,m}(u_1)\big|\le C_m|u_2-u_1|,\quad C_m:=\tan\theta_{m+1},\qquad\text{for any~$~u_1,u_2\ge R_n$.}
			\end{equation}
			In particular, if ${\Re \zeta_2=\Re \zeta_1}$, then the line segment joining $\zeta_1$ and~$\zeta_2$ lies entirely in~$U_{n,m}$. Since ${\Re F'>0}$ in the closure of~$U_{n,m}$, arguing as in the proof of the classical Noshiro\,--\,Warschawski Theorem (see e.g. \cite[Theorem~2.16]{Duren}) one can easily see that ${F(\zeta_2)\neq F(\zeta_1)}$.

			So we may suppose that ${\Re \zeta_2>\Re \zeta_1}$ and additionally, that the line segment~$I$ joining $\zeta_1$  with~$\zeta_2$ does not lie in the closure of~$U_{n,m}$. From the Lipschitz property~\eqref{EQ_fnm-Lip}, this is only possible in the case of ${|a|<C_m}$, where ${a\in\Real}$ stands for the slope of~$I$, meaning that  for a suitable ${b\in\Real}$ the segment~$I$ is contained in the graph of ${g(u):=au+b}$. Clearly, ${a\neq0}$. Thanks to the symmetry, we may suppose without loss of generality that ${a>0}$.
			Write ${\zeta_k=:u_k+iv_k}$, ${k=1,2}$. Since ${\zeta_1\in U_{n,m}}$, we have ${g(u_1)>-f_{n,m}(u_1)}$ and, due to monotonicity, ${g(u)>-f_{n,m}(u)}$ for all ${u\in[u_1,u_2]}$. Hence, $I$ can have intersection points with $\partial U_{n,m}$ only in the upper half-plane ${\{\zeta:\Im \zeta>0\}}$. It follows that the arc $\Gamma:=\{u+ih(u):u\in[u_1,u_2]\}$ of the graph of the continuous function $h(u):={\min\{g(u),f_{n,m}(u)\}}$  joining $\zeta_1$ to~$\zeta_2$ is contained in the closure of~$U_{n,m}$. Taking into account that $h$ is Lipschitz continuous with the same constant~$C_m$ and using~\eqref{EQ_F-prime}, it is not difficult to see that $\Re F(\zeta)$ changes strictly monotonically as $\zeta$ moves along~$\Gamma$. In particular, ${\Re F(\zeta_2)>\Re F(\zeta_1)}$.

			\StepPm{\ref{IT_known1-image-isogonal}}is completely standard. We use the notation introduced on the previous step. Setting ${F(\infty):=\infty}$, the function~$F$ extends to a continuous injective map of the closure of~$U_{n+1}$ regarded as a Jordan domain on the Riemann sphere for any ${n\in\Natural}$.
 Taking that by the hypothesis, ${\arg \big(F(\zeta)/\zeta\big)\to0}$ as ${\zeta\to\infty}$ non-tangentially, we may conclude that  ${F(U_{n+1})\supset A_{\theta_n,P_n}}$ for some ${P_n>0}$ large enough. It follows that
			$$
     F(U)~=~\bigcup_{n\in\Natural} F(U_n)~\supset~\bigcup_{n\in\Natural} F(U_{n+1})~\supset~\bigcup_{n\in\Natural} A_{\theta_n,P_n}
			$$
	and hence, $F(U)$ is isogonal at~$\infty$.

\StepP{\ref{IT_known1-F_inverse-isogonal}} We follow the same argument as in the previous step. Namely, given ${\theta\in(0,\pi/2)}$ choose some ${n\in\Natural}$ such that ${\theta_n>\theta}$. Then ${A_{\theta,P}\subseteq W_n:=F(U_n)}$ for some ${P>0}$. Therefore, for any $\zeta\in A_{\theta,P}$, we have $G(\zeta)=(F|_{U_n})^{-1}(\zeta)$. By Carath\'eodory's Extension Theorem (see e.g. \cite[Theorem~9.10]{Pombook75}), $F|_{U_n}$ extends to a homeomorphism~$F_n$ of the closure of $U_n$ onto the closure of $W_n$ in~${\C\cup\{\infty\}}$. Taking into account that ${F(x)\to\infty}$ as ${x\to+\infty}$, we have $F_n(\infty)=\infty$. Therefore, ${w:=G(\zeta)\to F_n^{-1}(\infty)=\infty}$ as ${W_n\supset A_{\theta,P}\ni \zeta\to\infty}$. It further follows that
			$$
			\lim_{A_{\theta,P}\ni \zeta\to\infty}\arg\frac{G(\zeta)}{\zeta}=\lim_{U_n\ni w\to\infty}\arg\frac{w}{F(w)}=0\quad\text{and~}~
			\lim_{A_{\theta,P}\ni \zeta\to\infty}\frac{G'(\zeta)}{|G'(\zeta)|}=\lim_{U_n\ni w\to\infty}\frac{|F'(w)|}{F'(w)}=1,
			$$
where the last equality holds by Lemma~\ref{LM_isogonality}\,\ref{IT_LM_isogonality(A)}.
		\end{proof}

We conclude the Appendix with two rather elementary lemmas, which are used in the proofs of our main results and which  we include for the sake of completeness.
		Let us denote by $\BH(\zeta_0,R)\coloneqq\{\zeta\in\UH\,:\, \hypdist_\UH(\zeta,\zeta_0)\leq R\}$  the closed hyperbolic disk in~$\UH$ of radius~${R>0}$ centered at~${\zeta_0\in\UH}$.
		\begin{lemma}\label{LM_qualitative}
			For any $R>0$ and any ${\varepsilon>0}$ there exists ${\delta\in(0,1)}$ such that for every function $F\in\Hol(\UH,\UH)$ satisfying ${F(1)=1}$ and ${F'(1)\in(1-\delta,1)}$, we have
			\begin{equation}\label{EQ_LM_qualitative}
				\max_{\zeta\in \BH(1,R)}\big|F'(\zeta)-1\big| \,<\, \varepsilon.
			\end{equation}
		\end{lemma}
		\begin{proof}
			Suppose on the contrary that the assertion does not hold. Then there exist $R>0$, ${\varepsilon>0}$, and a sequence $(F_n)\subseteq\Hol(\UH,\UH)$ with ${F_n(1)=1}$ such that ${F'_n(1)\to1}$ as ${n\to+\infty}$ and, at the same time, for all ${n\in\Natural}$
			\begin{equation}\label{EQ_contradiction}
				F'_n\big(\BH(1,R)\big)\not\subseteq\UD(1,\varepsilon)\coloneqq\{w\in\C\,:\, |w-1|<\varepsilon\}.
			\end{equation}

			By normality, $(F_n)$ has a subsequence $(F_{n_k})$ converging locally uniformly in~$\UH$. Denote by $F_*$ its limit. Clearly, ${F_*(1)=1}$. Therefore, ${F_*\in\Hol(\UH,\UH)}$ and ${F_*'(1)=1}$. The Schwarz\,--\,Pick Lemma for the half-plane implies $F_*=\id_\UH$. It follows that $F_n'\to F_*'\equiv1$ uniformly on~$\BH(1,R)$, which contradicts~\eqref{EQ_contradiction}.
		\end{proof}

			\begin{remark}\label{RM_l.u.-convergence}
				Let $D\subset\C$ be a domain. Recall that for a sequence of functions ${H_n:D_n\to\C}$ whose domains~$D_n$ do not necessarily contain~$D$, it is still possible to talk about locally uniform convergence in~$D$. Namely, such a sequence $(H_n)$ is said to \emph{converge} to a function ${H_*:D\to\C}$ \emph{locally uniformly in~$D$}, if for each compact set ${K\subset D}$ there exists ${n_0(K)\in\N}$ verifying the following two conditions: (i)~${K\subset D_n}$ for all ${n\ge n_0(K)}$, and (ii)~as~${m\to+\infty}$, ${H_{n_0(K)+m}\to H_*}$ uniformly on~$K$.
			\end{remark}

			\begin{lemma} \label{lem:DiscreteHypDisksToAngularLimit}
				Let $H$ be a complex-valued function in a domain ${U\subset\UH}$  and let ${\ell \in \C}$. Suppose that the sequence $(H_n)$, defined by $H_n(\zeta):={H\big( 2^n \zeta \big)}$ for $\zeta\in {D_n:=\{\zeta\colon 2^n\zeta\in U\}}$, converges locally uniformly in~$\UH$ to the constant function~$H_*\equiv\ell$.
				Then
				$$
				\angle \lim \limits_{\zeta \to \infty} H(\zeta)=\ell
				$$
				in the sense of Definition~\ref{DF_anglim-in-iso-domain}. In particular, $U$ is isogonal at~$\infty$.
			\end{lemma}
			\begin{proof}
				It is easy to show that $U$ is isogonal at~$\infty$. Indeed, given ${\theta\in(0,\pi/2)}$, consider the compact set $K:={\{\zeta\in\UH\colon|\arg\zeta|\le\theta,~1\le|\zeta|\le2\}}$. By the hypothesis, there is ${n_0\in\N}$ such that ${K\subset D_n}$ for all ${n\ge n_0}$. It follows that if we denote ${R:=2^{n_0}}$ and ${\Phi(\zeta):=2\zeta}$, then
				$$
				A_{\theta, R}~\subset~\bigcup_{n=n_0}^{+\infty}\Phi^{\circ n}(K)~\subset~ U,
				$$
				as desired.

				Furthermore, again by the hypothesis, for any given ${\varepsilon>0}$ there is ${m_0\in\N}$ such that ${|H_{n_0+m}(\zeta)-\ell|<\varepsilon}$ for all ${m\ge m_0}$ and all ${\zeta\in K}.$ It follows that
				$$
				|H(\zeta)-\ell|<\varepsilon\quad\text{for all $~\zeta\,\in\,\bigcup_{m=m_0}^{+\infty} \Phi^{\circ (n_0+m)}(K)~\supset~A_{\theta,\rho}~$~with~$~\rho:=2^{n_0+m_0}$.}
				$$
				Since $\theta\in(0,\pi/2)$ and ${\varepsilon>0}$ here are arbitrary, this argument proves that $\anglim_{\zeta\to\infty}H(\zeta)=\ell~$ in the sense of Definition~\ref{DF_anglim-in-iso-domain}.
			\end{proof}

	\end{document}